\documentclass[reqno]{amsart}
\usepackage{amsmath,amsthm,amstext,amssymb}
\usepackage{graphicx}

\theoremstyle{plain}
\newtheorem{theorem}{Theorem}[section]
\newtheorem{lemma}[theorem]{Lemma}
\newtheorem{question}[theorem]{Question}
\newtheorem{remark}[theorem]{Remark}
\newtheorem{corollary}[theorem]{Corollary}
\newtheorem{proposition}[theorem]{Proposition}

\newtheorem*{claim*}{Claim}

\newtheorem{example}[theorem]{Example}
\theoremstyle{definition}
\newtheorem{definition}[theorem]{Definition}

\newcommand{\betrag}[1]{\vert{#1}\vert}

\newcommand{\lub}{{\rm{lub}}}
\newcommand{\dom}[1]{{{\rm{dom}}(#1)}}

\newcommand{\cof}[1]{{{\rm{cof}}(#1)}}
\newcommand{\otp}[1]{{{\rm{otp}}\left(#1\right)}}
\newcommand{\ran}[1]{{{\rm{ran}}(#1)}}
\newcommand{\supp}[1]{{{\rm{supp}}(#1)}}
\newcommand{\rank}[2]{{\rm{rnk}}_{#2}({#1})}

\newcommand{\length}[1]{{\rm{lh}}({#1})}

\newcommand{\map}[3]{{#1}:{#2}\longrightarrow{#3}}
\newcommand{\Map}[5]{{#1}:{#2}\longrightarrow{#3};~{#4}\longmapsto{#5}}

\newcommand{\Set}[2]{\{{#1}~\vert~{#2}\}}
\newcommand{\seq}[2]{\langle{#1}~\vert~{#2}\rangle}

\newcommand{\goedel}[2]{{\prec}{#1},{#2}{\succ}}

\newcommand{\anf}[1]{{\text{``}\hspace{0.3ex}{#1}\hspace{0.3ex}\text{''}}}

\newcommand{\eins}{ {1{\rm\hspace{-0.5ex}l}} }

\newcommand{\HH}[1]{{\rm{H}}(#1)}

\newcommand{\Add}[2]{{\rm{Add}}({#1},{#2})}
\newcommand{\Coll}[2]{{\rm{Col}}({#1},{#2})}

\newcommand{\id}{{\rm{id}}}
\newcommand{\Lim}{{\rm{Lim}}}
\newcommand{\On}{{\rm{On}}}
\newcommand{\LL}{{\rm{L}}}
\newcommand{\ZFC}{{\rm{ZFC}}}
\newcommand{\GCH}{{\rm{GCH}}}

\newcommand{\MA}{{\rm{MA}}}

\newcommand{\PPP}{{\mathbb{P}}}
\newcommand{\QQQ}{{\mathbb{Q}}}
\newcommand{\RRR}{{\mathbb{R}}}

\newcommand{\VV}{{\rm{V}}}

\newcommand{\calA}{\mathcal{A}}

\newcommand{\calD}{\mathcal{D}}

\newcommand{\calS}{\mathcal{S}}
\newcommand{\calT}{\mathcal{T}}

 \newcommand{\CH}{({\rm{CH})}}

\usepackage[textsize=footnotesize,color=green!40, bordercolor=white]{todonotes}
\usepackage{hyperref} 

\usepackage[arrow, matrix, curve]{xy}

\DeclareFontFamily{OT1}{pzc}{}
\DeclareFontShape{OT1}{pzc}{m}{it}{<-> s * [1.100] pzcmi7t}{}
\DeclareMathAlphabet{\mathscr}{OT1}{pzc}{m}{it}

\title{Continuous Images of Closed Sets in Generalized Baire Spaces}

\thanks{The authors would like to thank the referee for the careful reading of the manuscript.} 

\author{Philipp L\"ucke}
\address{Philipp L\"ucke, Mathematisches Institut, Universit\"at Bonn,
Endenicher Allee 60, 53115 Bonn, Germany} 
\email{pluecke@math.uni-bonn.de}
\thanks{}

\author{Philipp Schlicht}
\address{Philipp Schlicht, Mathematisches Institut, Universit\"at Bonn,
Endenicher Allee 60, 53115 Bonn, Germany}
\email{schlicht@math.uni-bonn.de}


\begin{document}

\begin{abstract} 
Let $\kappa$ be an uncountable cardinal with $\kappa=\kappa^{{<}\kappa}$. Given a cardinal $\mu$, we equip the set ${}^\kappa\mu$ consisting of all functions from $\kappa$ to $\mu$ with the topology 
whose basic open sets consist of all extensions of partial functions of cardinality less than $\kappa$. 
We prove results that allow us to separate several classes of subsets of ${}^\kappa\kappa$ that consist of continuous images of closed subsets of spaces of the form ${}^\kappa\mu$.  
Important examples of such results are the following: 
(i) there is a closed subset of ${}^\kappa\kappa$ that is not a continuous image of ${}^\kappa\kappa$; (ii) there is an injective continuous image of ${}^\kappa\kappa$ that is not $\kappa$-Borel 
(i.e. that is not contained in the smallest algebra of sets on ${}^\kappa\kappa$ that contains all open subsets and is closed under $\kappa$-unions);  
(iii) the statement \anf{\emph{every continuous image of ${}^\kappa\kappa$ is an injective continuous image of a closed subset of ${}^\kappa\kappa$}} is independent of the axioms of $\ZFC$; and 
(iv) the axioms of $\ZFC$ do not prove that the assumption ${\anf{2^\kappa>\kappa^+}}$ implies the statement \anf{\emph{every closed subset of ${}^\kappa\kappa$ is a continuous image of 
${}^\kappa(\kappa^+)$}} or its negation. 
\end{abstract}

\maketitle

\setcounter{tocdepth}{1}
\tableofcontents

\section{Introduction}\label{section:Intro}

Let $\kappa$ be an infinite regular cardinal. Given a cardinal $\mu$, we equip the set ${}^\kappa\mu$ consisting of all functions $\map{x}{\kappa}{\mu}$ with the topology 
whose basic open sets are of the form $$ N_s ~ = ~ \Set{x\in {}^\kappa\mu}{s\subseteq x},$$ where $s$ is an element of the set ${}^{{<}\kappa}\mu$ of all functions $\map{t}{\alpha}{\mu}$ with $\alpha<\kappa$. 
We let $\mathbf{\Sigma}^0_1(\kappa)$ denote the class of all open subsets of ${}^\kappa\kappa$ and we use $B(\kappa)$ to denote the class of all $\kappa$-Borel subsets of ${}^\kappa\kappa$, 
i.e. the class of all subsets that are contained in the smallest algebra of sets on ${}^\kappa\kappa$ that contains all open subsets and is closed under $\kappa$-unions. 
A subset of ${}^\kappa\kappa$ is a $\mathbf{\Sigma}^1_1$-subset if it is equal to the projection of a closed subset of ${}^\kappa\kappa\times{}^\kappa\kappa$ 
and it is a $\mathbf{\Delta}^1_1$-subset if both the set itself and its complement are $\mathbf{\Sigma}^1_1$-subsets.  
Since the graph of a continuous function $\map{f}{C}{{}^\kappa\kappa}$ with $C\subseteq{}^\kappa\kappa$ closed is again a closed subset of ${}^\kappa\kappa\times{}^\kappa\kappa$ 
and the spaces ${}^\kappa\kappa$ and ${}^\kappa\kappa\times{}^\kappa\kappa$ are homeomorphic, it follows that a subset $A$ of ${}^\kappa\kappa$ is a $\mathbf{\Sigma}^1_1$-subset 
if and only if it is equal to the continuous image of a closed subset of ${}^\kappa\kappa$ 
(in the sense that there is a closed subset $C$ of ${}^\kappa\kappa$ and a function $\map{f}{C}{{}^\kappa\kappa}$ such that $A=\ran{f}$ and $f$ is continuous with respect to the subspace topology on $C$).

If $\kappa$ is an uncountable cardinal with $\kappa=\kappa^{{<}\kappa}$, then it is well-known (see, for example, {\cite[Section 2]{MR2987148}}) that the class of $\mathbf{\Sigma}^1_1$-subsets 
of ${}^\kappa\kappa$ is equal to the class of subsets of ${}^\kappa\kappa$ that are definable over the structure $(\HH{\kappa^+},\in)$ by a $\Sigma_1$-formula with parameters. 
This shows that many interesting and important subsets of ${}^\kappa\kappa$ are equal to continuous images of closed subsets of spaces of the form ${}^\kappa\mu$. 
We present two examples of prominent $\mathbf{\Sigma}^1_1$-subsets that are contained in smaller classes of continuous images.

\begin{example} 
 The club filter 
 \begin{equation*} 
  \mathrm{Club}_\kappa ~ = ~ \Set{x\in{}^\kappa\kappa}{\textit{$\exists C\subseteq\kappa$ club} ~ \forall\alpha\in C ~ x(\alpha)=0} 
 \end{equation*} 
 is a continuous image of the space ${}^\kappa\kappa$. Let $T$ denote the set of all pairs $(s,t)$ in ${}^{\gamma}\kappa\times{}^{\gamma}2$ such that $\gamma<\kappa$ is a limit ordinal, $t(\alpha)=0$ implies $s(\alpha)=0$ 
 for all $\alpha<\gamma$, and the set $\Set{\alpha<\gamma}{t(\alpha)=0}$ is a closed unbounded subset of $\gamma$. 
 We order $T$ by componentwise inclusion. Then $T$ is a tree of height $\kappa$ that is closed under increasing sequences of length less than $\kappa$ 
 and every node in $T$ has $\kappa$-many direct successors, because every limit ordinal of countable cofinality in the interval $(\length{s},\kappa)$ gives rise to a distinct direct successor of a node $(s,t)$ in $T$. 
 This shows that $T$ is isomorphic to the tree ${}^{<\kappa}\kappa$.  
 If we equip the set $$[T] ~ = ~ \Set{(x,y)\in{}^\kappa\kappa\times{}^\kappa 2}{\forall\alpha<\kappa ~ \exists\alpha<\beta<\kappa ~ (x\restriction\beta,y\restriction\beta)\in T}$$ 
 with the topology whose basic open sets consist of all extensions of elements of $T$, then we obtain a topological space homeomorphic to ${}^\kappa\kappa$. 
 Since the projection $\map{p}{[T]}{{}^\kappa\kappa}$ onto the first coordinate is continuous and $\ran{p}=\mathrm{Club}_\kappa$, we can conclude that the set $\mathrm{Club}_\kappa$ is equal 
 to a continuous image of ${}^\kappa\kappa$. 
\end{example}

Given an infinite regular cardinal $\kappa$ and cardinals $\mu_0,\ldots,\mu_n$, we call a subset $T$ of the product ${}^{{<}\kappa}\mu_0\times\ldots\times{}^{{<}\kappa}\mu_n$ a \emph{subtree} if $\length{t_0}=\ldots=\length{t_n}$ and 
the tuple $(t_0\restriction\alpha,\ldots,t_n\restriction\alpha)$ is an element of $T$ whenever $(t_0,\ldots,t_n)\in T$ and $\alpha<\length{t_0}$. 
We use $\leq$ to denote the natural tree-ordering on such a subtree $T$, i.e. if $s=(s_0,\ldots,s_n)$ and $t=(t_0,\ldots,t_n)$ are nodes in $T$, then we write $s\leq t$ to denote that $s_i\subseteq t_i$ holds for all $i\leq n$.

Given a subtree $T$ of ${}^{{<}\kappa}\mu_0\times\ldots\times{}^{{<}\kappa}\mu_n$, we say that an element $(x_0,\ldots,x_n)$ of ${}^\kappa\mu_0\times\ldots\times{}^\kappa\mu_n$ is a \emph{cofinal branch through $T$} 
if $(x_0\restriction\alpha,\ldots,x_n\restriction\alpha)\in T$ for every $\alpha<\kappa$. It is easy to see that a subset $A$ of ${}^\kappa\mu_0\times\ldots\times{}^\kappa\mu_n$ is closed with respect to the topology 
introduced above if and only if it is equal to the set $[T]$ of all cofinal branches though some subtree $T$ of ${}^{{<}\kappa}\mu_0\times\ldots\times{}^{{<}\kappa}\mu_n$.

\begin{example}\label{example:trees}
 Suppose that $\map{f}{{}^{{<}\kappa}\kappa}{\kappa}$ is a bijection. Let $T_{\kappa}$ denote the set of all $x\in {}^{\kappa}2$ such that $x\circ f$ is the characteristic function of a subtree $T$ of ${}^{{<}\kappa}\kappa$ 
 with $[T]\neq\emptyset$.  
 The results of {\cite[Section 2]{MR1242054}} show that $T_\kappa$ is a  
 $\mathbf{\Sigma}^1_1$-complete subset of ${}^\kappa\kappa$, i.e. if $A$ is a nonempty $\mathbf{\Sigma}^1_1$-subset of ${}^\kappa\kappa$, then there is a continuous function $\map{f}{{}^\kappa\kappa}{{}^\kappa\kappa}$ with $A=f^{{-}1}[T_\kappa]$.

 The set $T_\kappa$ is also equal to a continuous image of ${}^{\kappa}\kappa$. 
 To see this, let $S$ denote the set of all pairs $(t,u)$ in ${}^{\gamma}2\times{}^{\gamma}\kappa$ with $\gamma<\kappa$ such that the set $$T(t) ~ = ~ \Set{s\in{}^{{<}\kappa}\kappa}{f(s)<\gamma, ~ (t\circ f)(s)=1}$$ is 
 a subtree of ${}^{<\kappa}\kappa$ with $u\restriction\bar{\gamma}\in T(t)$ for every $\bar{\gamma}<\gamma$. If we order the set $S$ by componentwise inclusion, then the resulting tree is isomorphic to ${}^{<\kappa}\kappa$. 
 As above, we equip the set $$[S] ~ = ~ \Set{(x,y)\in{}^\kappa 2\times{}^\kappa\kappa}{\forall\alpha<\kappa ~ \exists\alpha<\beta<\kappa ~ (x\restriction\beta,y\restriction\beta)\in S}$$ with the topology induced by $S$ 
 and obtain a space  homeomorphic to ${}^\kappa\kappa$. Since $T_\kappa$ is equal to the projection of  $[S]$, this set is equal to a continuous image of ${}^\kappa\kappa$.  
\end{example}

In this paper, we study the provable and relative consistent statements about the relationships between different classes of such continuous images in the case where $\kappa$ is an uncountable cardinal with $\kappa=\kappa^{{<}\kappa}$.  
We will consider the following classes of subsets.

\begin{itemize}
 \item The class $C^\kappa$ of continuous images of ${}^\kappa\kappa$. 

 \item The class $\mathbf{\Sigma}^1_1(\kappa)$ of continuous images of closed subsets of ${}^\kappa\kappa$. 

 \item The class $I^\kappa$ of continuous injective images of ${}^\kappa\kappa$. 

 \item The class $I^\kappa_{cl}$ of continuous injective images of closed subsets of ${}^\kappa\kappa$. 
\end{itemize}
It will turn out that the following class is also important for this analysis. 
\begin{itemize}
 \item Let $M$ be an inner model of set theory and $n<\omega$. We define $S_n^{M,\kappa}$ to be the class of all subsets $A$ of ${}^\kappa\kappa$ such that 
  \begin{equation*}
   A ~ = ~ \Set{x\in{}^\kappa\kappa}{M[x,y]\models\varphi(x,y)}
  \end{equation*}
  for some $y\in{}^\kappa\kappa$ and a $\Sigma_n$-formula $\varphi(u,v)$, where $M[x,y]=\bigcup_{z\in M}\LL[x,y,z]$ is the smallest transitive model of set theory containing $M\cup\{x,y\}$. 
\end{itemize}


To motivate our work, we start by briefly reviewing the relations between these classes in the classical setting $\anf{\kappa=\omega}$. 
\begin{itemize}
 \item The classes $B(\omega)$ and $I^\omega_{cl}$ coincide (see {\cite[15.3]{MR1321597}}). 

 \item Since every set in $I^\omega$ has no isolated points, the class $I^\omega$ is a proper subclass of $B(x)$ that contains every non-empty open set. Moreover, every Borel subset of ${}^\omega\omega$ is equal 
   to the union of an element of $I^\omega$ and a countable set (this follows from {\cite[6.4]{MR1321597}} and {\cite[13.1]{MR1321597}}). 

 \item A nonempty subset of ${}^{\kappa}\kappa$ is in $\mathbf{\Sigma}^1_1(\omega)$ if and only if it is in $C^\omega$ (see {\cite[2.8]{MR1321597}}). 

 \item Given an inner model $M$ of $\ZFC$, the class $S_1^{M,\omega}$ coincides with the class of all $\mathbf{\Sigma}^1_2$-subsets of ${}^\omega\omega$, i.e. with the class of projections of complements of $\mathbf{\Sigma}^1_1$-subsets of 
  ${}^\kappa\kappa\times{}^\kappa\kappa$ (this follows from {\cite[Lemma 25.20]{MR1940513}} and {\cite[Lemma 25.25]{MR1940513}}). 
\end{itemize}


\begin{figure}[ht]
 \begin{equation*}
\begin{xy}
\xymatrix{
 \mathbf{\Sigma}^0_1(\omega)   \ar[r] &  I^\omega  \ar[r] &  B(\omega)  \ar@{-}@<0.2ex>[r] \ar@{-}@<-0.2ex>[r] &  I^\omega_{cl}  \ar[r] &  \mathbf{\Sigma}^1_1(\omega)  \ar@{-}@<0.2ex>[r] \ar@{-}@<-0.2ex>[r] &   C^\omega  \ar[r] &  S_1^{M,\omega}  \\ 
}
\end{xy}
\end{equation*} 
\caption{The provable and consistent relations between the considered classes in the case $\anf{\kappa=\omega}$.}
\label{figure:Classical}
\end{figure}


We summarize the above relationships in Figure \ref{figure:Classical} with the help of a \emph{complete diagram of provable and consistent relations}. 
Given two classes $C$ and $D$ of subsets, we use a solid arrow $A\longrightarrow B$ from $A$ to $B$ to indicate that it is provable from the axioms of $\ZFC$ that 
every non-empty subset in $A$ is an element of $B$ and a dashed arrow $A\dasharrow B$ to indicate that this inclusion is relatively consistent with these axioms. 
Such a diagram is \emph{complete} if the transitive hull of the displayed implications contains all possible implications.


\begin{figure}[ht]
 \begin{equation*}
\begin{xy}
\xymatrix{
 S_1^{\LL,\kappa} \ar[rr] & &  I^\kappa_{cl} \ar[rr]    & & \mathbf{\Sigma}^1_1(\kappa) \ar@{-->}@/_ 4ex/[llll]_{\LL[x]} \\ 
 B(\kappa)   \ar[u]  & &    & & \\
 \mathbf{\Sigma}^0_1(\kappa) \ar[u]   \ar[rr]  & &  I^\kappa  \ar[uu]  \ar[rr]   & & C^\kappa \ar[uu]    \\
}
\end{xy}
\end{equation*} 
\caption{The provable and consistent relations in the case where $\kappa$ is an uncountable cardinal with $\kappa=\kappa^{{<}\kappa}$.}
\label{figure:Results}
\end{figure}
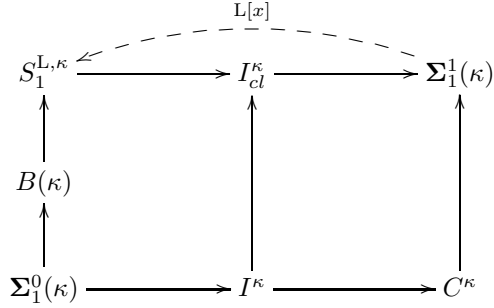


The results of this paper will show that the above classes relate in a fundamentally different way if $\kappa$ is an uncountable cardinal with $\kappa=\kappa^{{<}\kappa}$. 
These results are summarized by Figure \ref{figure:Results} and they will show that this diagram is also complete. 


An important result in the above analysis is the observation that the class of continuous images of ${}^\kappa\kappa$ does not contain every nonempty closed subset of ${}^\kappa\kappa$. 
We will show that the closed subset constructed in the proof of this result is equal to a continuous image of the space ${}^\kappa(\kappa^+)$. 
Motivated by this fact, we also investigate the following classes in the case where $\mu$ is a cardinal with $\kappa<\mu<2^\kappa$. 
\begin{itemize}
 \item The class $C^{\kappa,\mu}$ of continuous images of ${}^\kappa\mu$. 

 \item The class $C^{\kappa,\mu}_{cl}$ of continuous images of closed subsets of ${}^\kappa\mu$. 
\end{itemize}

As above, we first discuss known results about the relationships between these classes in the countable case.

\begin{itemize}
 \item A nonempty subset of ${}^{\kappa}\kappa$ is in $C^{\omega,\mu}$ if and only if it is in $C^{\omega,\mu}_{cl}$ (see {\cite[2.8]{MR1321597}}). 

 \item Every $\mathbf{\Sigma}^1_2$-subset of ${}^\omega\omega$ is equal to the projection of a subtree of ${}^\omega\omega\times{}^\omega\omega_1$ (see {\cite[38.9]{MR1321597}}). 
       In particular, $\mathbf{\Sigma}^1_1(\omega)$ is a proper subclass of $C^{\omega,\omega_1}$.

 \item If every uncountable $\mathbf{\Sigma}^1_2$-subset of ${}^\omega\omega$ contains a perfect subset, then there is a set in $C^{\omega,\omega_1}$ that is not a $\mathbf{\Sigma}^1_2$-subset. 

 \item If every subset of ${}^\omega\omega$ of cardinality $\omega_1$ is a $\mathbf{\Sigma}^1_2$-subset, then every set in the class $C^{\omega,\omega_1}$ is a $\mathbf{\Sigma}^1_2$-subset 
  (see the proof of Proposition \ref{proposition:BigTreesDefinable} for a similar argument).  By using \emph{almost disjoint coding} (see \cite{MR0289291} and {\cite[Section 1]{MR0465866}}), 
  the above assumption can be seen to hold in every model of $\MA_{\omega_1}+\neg\CH+\anf{\omega_1=\omega_1^\LL}$ (see {\cite[Section 3.2]{MR0270904}}). 
\end{itemize}


\begin{figure}[ht]
 \begin{equation*}
\begin{xy}
\xymatrix{
 C^{\omega,\mu} \ar@{-}@<0.2ex>[rrr] \ar@{-}@<-0.2ex>[rrr] & & & C^{\omega,\mu}_{cl} ~ \ar@{-->}@<0.5ex>[dd]^{\VV\models\MA_{\omega_1}+\anf{\omega_1=\omega_1^\LL}} \\
 & & &  \\
 \mathbf{\Sigma}^1_1(\omega) ~ \ar[uu]        \ar[rrr]  & & &  S_1^{\LL,\omega} ~ \ar@<0.5ex>[uu] \\
}
\end{xy}
\end{equation*} 
\caption{The provable and consistent relations between the considered classes in the case where $\anf{\kappa=\omega}$  and $\mu$ is a cardinal with $\omega<\mu<2^\omega$.}
\label{figure:ClassicalHigher}
\end{figure}
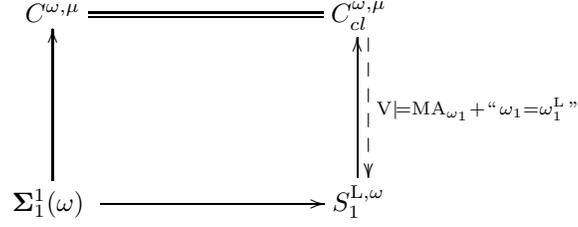


The complete diagram shown in Figure \ref{figure:ClassicalHigher} summarizes these results.


Analogous to the above results, these classes behave quite differently in the case where $\kappa$ is an uncountable cardinal with $\kappa=\kappa^{{<}\kappa}$ and $\mu$ is a cardinal with $\kappa<\mu<2^\kappa$.
The results of this paper will show that the diagram shown in Figure \ref{figure:ResultsMu} is complete. 


\begin{figure}[ht]
 \begin{equation*}
\begin{xy}
\xymatrix{
  C^{\kappa,\mu} \ar@<-0.5ex>[rrr]   & & & C^{\kappa,\mu}_{cl} \ar@{-->}@<0.5ex>[dd]^{\VV\models\rm{BA}(\kappa)+\anf{\kappa^+=(\kappa^+)^\LL}}  \ar@{-->}@<-0.5ex>[lll]_{\VV^{\Add{\kappa}{\theta}}} \\ 
    & & & \\
  C^\kappa  \ar[uu]  \ar[rrr]  & & & \mathbf{\Sigma}^1_1(\kappa) \ar@<0.5ex>[uu]    \\
}
\end{xy}
\end{equation*} 
\caption{The provable and consistent relations between the considered classes in the case where $\kappa$ is an uncountable cardinal with $\kappa=\kappa^{{<}\kappa}$ and $\mu$ is a cardinal with $\kappa<\mu<2^\kappa$.}
\label{figure:ResultsMu}
\end{figure}
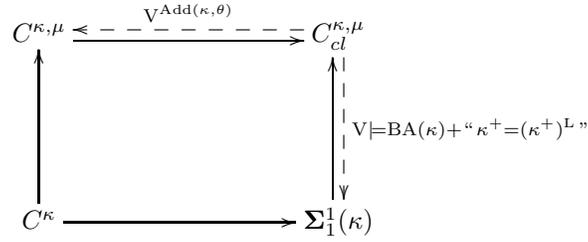


In the remainder of this section, we present the results that imply that the above diagrams contain all provable and consistent implications.


\subsection{\texorpdfstring{The class $C^\kappa$ of continuous images of ${}^\kappa\kappa$}{The class C-kappa}}\label{subsection:CKappa}

An important topological property of spaces of the form ${}^\omega\mu$ is the fact that non-empty closed subsets are retracts of the whole space (see  {\cite[2.8]{MR1321597}}), i.e. given a closed subset $A$ of ${}^\omega\mu$ 
there is a continuous (even Lipschitz) surjection $\map{f}{{}^\omega\mu}{A}$ such that $f\restriction A=\id_A$. The proof of this statement generalizes to higher cardinalities for a certain class of closed subsets.

 Given an infinite regular cardinal $\kappa$ and cardinals $\mu_0,\ldots,\mu_n,\nu$,  we say that a subtree $T$ of    ${}^{{<}\kappa}\mu_0\times\ldots\times{}^{{<}\kappa}\mu_n$ is \emph{${<}\nu$-closed} if for every $\lambda<\nu$ and every  $\leq$-increasing sequence $\seq{(t_0^\alpha,\ldots,t^\alpha_n)}{\alpha<\lambda}$ in $T$, the tuple   $(\bigcup_{\alpha<\lambda}t^\alpha_0,\ldots,\bigcup_{\alpha<\lambda} t^\alpha_n)$ is an element of $T$.  
We call a node $t$ of such a tree $T$ an \emph{end node} if it is maximal in $T$ with respect to the componentwise ordering.

\begin{proposition}\label{proposition:ClosedRetract}
 Let $\kappa$ be an infinite regular cardinal, $\mu_0,\ldots,\mu_n$ be cardinals and $T$ be a subtree of ${}^{{<}\kappa}\mu_0\times\ldots\times{}^{{<}\kappa}\mu_n$. 
 If $T$ is a ${<}\kappa$-closed tree without end nodes, then the closed set $[T]$ is a retract of ${}^\kappa\mu_0\times\ldots\times{}^\kappa\mu_n$. 
\end{proposition}

\begin{proof}
 Note that our assumptions imply that there is a cofinal branch through every node of $T$. We define $\partial T$ to be the set 
 \begin{equation}\label{equation:DefBorderT}
  \Set{(s_0,\dots,s_n)\in({}^\gamma\mu_0\times\dots\times{}^\gamma\mu_n)\setminus T}{\gamma<\kappa,  \forall\alpha<\gamma (s_0\restriction\alpha,\dots,s_n\restriction\alpha)\in T}.   
 \end{equation}
 Then our assumptions imply that every element $s$ of $\partial T$ is the direct successor of an element $t_s$ of $T$ and 
 there is a cofinal branch $x_s$ through $T$ with $t_s\subseteq x_s$. If $y=(y_0,\dots,y_n)\in({}^\kappa\mu_0\times\dots\times{}^\kappa\mu_n)\setminus[T]$, then 
 there is a unique $\alpha<\kappa$ such that $(y_0\restriction\alpha,\dots,y_n\restriction\alpha)\in\partial T$ and we let $s_y$ denote this tuple.  
 This allows us to define a surjection $\map{f}{{}^\kappa\mu_0\times\dots\times{}^\kappa\mu_n}{[T]}$ by setting $f(x)=x$ if $x\in[T]$ and $f(y)=x_{s_y}$ otherwise. It is easy to see that $f$ is continuous.
\end{proof}

The following observation shows that the above topological property does not hold for all closed subsets of ${}^\kappa\mu$ if $\kappa$ is regular and uncountable.

\begin{proposition} \label{NoRetraction} 
 Suppose that $\kappa$ is an uncountable regular cardinal and $\mu>1$ is a cardinal. Let $A$ denote the set of all $x$ in ${}^\kappa\mu$ such that $x(\alpha)=0$ for only finitely many $\alpha<\kappa$.
 Then $A$ is a closed subset of ${}^\kappa\mu$ that is not a retract of ${}^\kappa\mu$. 
\end{proposition} 

\begin{proof}
 Let $T$ denote the subtree of ${}^{{<}\kappa}\mu$ consisting of all $t\in {}^{{<}\kappa}\mu$ such that  $t(\alpha)=0$ for only finitely many $\alpha\in dom(t)$. 
 Since $\cof{\kappa}>\omega$, we have $A=[T]$ and this shows that $A$ is a closed subset of ${}^\kappa\mu$.

 Suppose that $\map{f}{{}^{\kappa}\mu}{A}$ is a retraction onto $A$. 
 We inductively construct sequences $\seq{x_n\in A}{n<\omega}$ and $\seq{\gamma_n<\kappa}{n<\omega}$ such that $x_n(\gamma_n)=0$, $\gamma_n<\gamma_{n+1}$ and $x_{n}\restriction\gamma_{n+1}=x_{n+1}\restriction\gamma_{n+1}$ for all $n<\omega$. 
 Let $\gamma_0=0$ and $x_0$ be an arbitrary element of $A$ with $x_0(0)=0$. Now assume that $x_n$ and $\gamma_n$ are already constructed. Then we find $\gamma_n<\gamma_{n+1}<\kappa$ with 
 $f[N_{x_n\restriction \gamma_{n+1}}]\subseteq N_{x_n\restriction (\gamma_n+1)}$. Pick $x_{n+1}\in A$ with $x_{n+1}\restriction\gamma_{n+1}=x_n\restriction\gamma_{n+1}$ and $x_{n+1}(\gamma_{n+1})=0$.

 Now pick $x_\omega\in{}^\kappa\mu$ with $x_n\restriction\gamma_{n+1}\subseteq x_\omega$ for all $n<\omega$ and set $\gamma_\omega=\sup_{n<\omega}\gamma_n$. 
 Given $n<\omega$, we have $f(x_\omega)\in N_{x_n\restriction\gamma_n+1}$ and therefore $f(x_\omega)(\gamma_n)=x_n(\gamma_n)=0$. Hence $f(x_{\omega})\notin A$, contradicting our assumption on $f$. 
\end{proof}

Since every non-empty closed subset of the Baire space ${}^\omega\omega$ is a retract of ${}^\omega\omega$, it follows that every closed subset of ${}^\omega\omega$ is a continuous image of ${}^\omega\omega$ and this 
implies that the class $C^\omega$ is equal to the class of all $\mathbf{\Sigma}^1_1$-subsets of ${}^\omega\omega$.
It is natural to ask whether these classes are also identical if $\kappa$ is an uncountable regular cardinal. In Section \ref{Section:thin}, we will answer this question in the negative by proving the following result.

\begin{theorem}\label{theorem:Main1}
 Let $\kappa$ be an uncountable cardinal with $\kappa=\kappa^{{<}\kappa}$. Then there is a closed non-empty subset of ${}^\kappa\kappa$ that is not the continuous image of ${}^\kappa\kappa$.
\end{theorem}

This result has the following corollary that shows that the class $C^\kappa$ lacks important closure properties.

\begin{corollary} 
 The class $C^\kappa$ is not closed under finite intersections. 
\end{corollary} 

\begin{proof} 
 Let $A=[T]$ be the closed subset provided by Theorem \ref{theorem:Main1}. Define $\partial T$ as in (\ref{equation:DefBorderT}) in the proof of Proposition \ref{proposition:ClosedRetract}. 
 Given $i<2$, define $T_i$ to be $T$ together with the set of all extensions of elements of $\partial T$ by strings of $i$'s of length less than $\kappa$, i.e. we define 
 \begin{equation*}
  T_i ~ = ~ T ~ \cup ~ \Set{t\in{}^{{<}\kappa}\kappa}{\exists\alpha\leq\length{t} ~ [t\restriction\alpha\in\partial T\wedge \forall\alpha\leq\beta<\length{t} ~ t(\beta)=i]}. 
 \end{equation*}
 Then $T_i$ is a ${<}\kappa$-closed subtree of ${}^\kappa\kappa$ without end nodes. 
 By Proposition \ref{proposition:ClosedRetract}, the subset $[T_i]$ is an element of $C^\kappa$. Our construction ensures that the intersection $[T_0]\cap[T_1]$ is equal to the set $A$ 
 and therefore is not an element of $C^\kappa$. 
\end{proof}


\subsection{\texorpdfstring{The class $I^\kappa$ of continuous injective images of ${}^\kappa\kappa$}{The class I-kappa}}

We discuss results separating the class $I^\kappa$ from the other classes introduced above. 
Since every image of ${}^{\kappa}\kappa$ under a continuous injective map has no isolated points, it follows that the classes  $C^{\kappa}$ and $I_{cl}^{\kappa}$ are both not contained in $I^{\kappa}$.

\begin{proposition} 
 If $\kappa$ is an uncountable cardinal with $\kappa=\kappa^{{<}\kappa}$, then every non-empty open subset of ${}^\kappa\kappa$ is equal to an injective continuous image of ${}^\kappa\kappa$.
\end{proposition} 

\begin{proof} We can write every open subset $A$ as a disjoint union $A=\bigcup_{\alpha<\kappa}N_{s_{\alpha}}$ of basic open sets, where $s_{\alpha}\in {}^{<\kappa}\kappa$ for  $\alpha<\kappa$. Next, we choose homeomorphisms $\map{f_{\alpha}}{ N_{s_{\alpha}}}{{}^{\kappa}\kappa}$ and combine them to a homeomorphism $\map{f}{{}^{\kappa}\kappa}{A}$. 
\end{proof}

To separate $I^\kappa$ from the class of all $\kappa$-Borel subsets of ${}^\kappa\kappa$, we need to introduce an important regularity property of subsets of ${}^\kappa\kappa$.  
We say that a subset $A$ of ${}^\kappa\kappa$ \emph{has the $\kappa$-Baire property} if there is an open subset $U$ of ${}^\kappa\kappa$ and a sequence $\seq{N_\alpha}{\alpha<\kappa}$ 
of nowhere dense subsets of ${}^\kappa\kappa$ such that the symmetric difference $A_\Delta U$ is a subset of $\bigcup_{\alpha<\kappa}N_\alpha$. Every $\kappa$-Borel subset of ${}^\kappa\kappa$ 
has the $\kappa$-Baire property (see {\cite{MR1880900}}). It is consistent that every $\mathbf{\Delta}^1_1$-subset has this property and 
it is also consistent that there is a $\mathbf{\Delta}^1_1$-subset of ${}^{\kappa}\kappa$ which does not have the $\kappa$-Baire property (see {\cite[Theorem 49]{FHK}}). We will prove the following result in Section \ref{section:I kappa}.

\begin{theorem} \label{non-separable I-kappa sets} 
 If $\kappa$ is an uncountable cardinal with $\kappa=\kappa^{{<}\kappa}$, then there is a sequence $\seq{A_\gamma\subseteq{}^\kappa\kappa}{\gamma<2^\kappa}$ of pairwise disjoint injective continuous images of ${}^\kappa\kappa$ 
 such that $A_\gamma$ and $A_{\delta}$ cannot be separated by sets with the $\kappa$-Baire property for all $\gamma<\delta<2^\kappa$. 
\end{theorem}

\begin{corollary} \label{non-Borel I-kappa sets} 
 If $\kappa$ is an uncountable cardinal with $\kappa=\kappa^{{<}\kappa}$, then there is a continuous injective image of ${}^\kappa\kappa$ that is not a $\kappa$-Borel subset of ${}^\kappa\kappa$. \qed
\end{corollary}

In Section \ref{section:I kappa}, we will present a variation of the proof of the above theorem that will allow us to prove the following surprising implication.

\begin{theorem} \label{special trees Baire} 
 Let $\nu$ be an infinite cardinal with $\nu=\nu^{{<}\nu}$ and $\kappa=\nu^+=2^\nu$. Assume that every $\kappa$-Aronszajn tree $T$ that does not contain a $\kappa$-Souslin subtree is 
 special.\footnote{Let $\kappa$ be an uncountable regular cardinal. A \emph{$\kappa$-Aronszajn tree} is a tree of height $\kappa$ without $\kappa$-branches and the 
 property that every level has cardinality less than $\kappa$. A $\kappa$-Aronszajn tree is a \emph{$\kappa$-Souslin tree} if every antichain in this tree has cardinality less than $\kappa$. If $\nu$ is an infinite cardinal, 
 then a tree $T$ of height $\nu^+$ is \emph{special} if there is a function $\map{s}{T}{\nu}$ that is injective on chains in $T$.}
 Then there is a $\mathbf{\Delta}^1_1$-subset of ${}^\kappa\kappa$ that does not have the $\kappa$-Baire property.
\end{theorem}

The assumptions of the above theorem are known to be consistent in the case {\anf{$\nu=\omega$}} (see {\cite[Theorem 4.1]{MR788070}}). 
If $\nu$ is uncountable, then this assumption implies that $\square(\kappa)$ fails and $\kappa$ is weakly compact in $\LL$ (see {\cite[Corollary 6.3.3]{MR2355670}} and {\cite[Theorem 3]{MR929410}}).


\subsection{\texorpdfstring{The class $I^\kappa_{cl}$ of continuous injective images of closed subsets of ${}^\kappa\kappa$}{The class I-kappa-cl}}

By the above results, the class $I^\kappa_{cl}$ is neither contained in the class $B(\kappa)$ nor in the class of nonempty sets in $C^\kappa$. We present results that imply the remaining implications shown in the diagram of Figure \ref{figure:Results}.

\begin{lemma}\label{lemma:BorelIsInIKappaClosed}
 Every $\kappa$-Borel subset of ${}^{\kappa}\kappa$ is  a continuous injective image of a closed subset of ${}^\kappa\kappa$. 
\end{lemma} 

\begin{proof} 
 We can code a Borel set $B\subseteq {}^{\kappa}\kappa$ by a well-founded tree $T\subseteq {}^{{<}\omega}\kappa$ which has basic open sets $N_{s_t}$ attached to the end nodes $t$, 
 and labels $l_t\in\{c,u\}$ (for \emph{complement} and \emph{union}) at each non-end node $t$ such that every node with label $c$ has a unique successor in $T$.

Fix some well-ordering $\prec$ of $T$. We call a pair $(y,z)$  \emph{correct} if the following statements hold. 
 \begin{enumerate}
  \item $\map{y}{T}{2}$ is a function with $y(\emptyset)=1$, $l_{t_0}=c$ implies $y(t_0)=1-y(t_1)$ whenever $t_1$ is the unique successor of $t_0$ in $T$, and $l_{t_0}=u$ implies that $y(t_0)=1$ holds if and only if there is a direct successor $t_1$ of $t_0$ in $T$ with $y(t_1)=1$. 

  \item $\map{z}{T}{T}$ is a function such that $z(t)$ is the $\prec$-minimal successor $t'$ of $t$ with $y(t')=1$ whenever $l_t=u$ with $y(t)=1$ and $z(t)=t$ otherwise. 
 \end{enumerate}

 Let $\mathrm{End}_T$ denote the set of end nodes of $T$. 
 Define $C$ to be the set  $$ \Set{(x,y,z)\in {}^{\kappa}\kappa\times {}^{T}2\times{}^T T}{\textit{$(y,z)$ is correct} ~ \wedge ~ \forall t\in \mathrm{End}_T\ [y(t)=1\Longleftrightarrow x\in N_{s_t}]}.$$ 

 If we equip the set $C$ with the initial segment topology, then the resulting topological space is homeomorphic to a closed subset of the space ${}^\kappa\kappa$ and the projection $\map{p}{C}{{}^\kappa\kappa}$ onto the first coordinate is continuous. 
 Moreover, $B=p[C]$ and this projection is injective, because the pair $(y,z)$ is uniquely determined by $x$ for every triple $(x,y,z)\in C$. 
\end{proof}

With the help of the above lemma, we will prove the following result that will allow us to show that it is consistent with the axioms of $\ZFC$ that every $\mathbf{\Sigma}^1_1$-set is equal 
to a continuous injective image of a closed subset of ${}^\kappa\kappa$.

\begin{theorem}\label{theorem:SKappa1LinIKappaClosed}
 Let $\kappa$ be an uncountable cardinal with $\kappa=\kappa^{{<}\kappa}$. Then every subset in the class $S_1^{\LL,\kappa}$ is equal to an injective continuous image of a closed subset of ${}^\kappa\kappa$. 
\end{theorem}

The above statement now follows from the following observation.

\begin{proposition} \label{proposition: Sigma-1-1 and S-1 are equal in L} 
 Assume $V=\LL[x]$ with $x\subseteq \kappa$ for some uncountable cardinal $\kappa$ with $\kappa=\kappa^{{<}\kappa}$. Then the classes $S_1^{\LL,\kappa}$ and $\mathbf{\Sigma}^1_1(\kappa)$ coincide. 
\end{proposition} 

\begin{proof}
 By the above remarks, a subset of ${}^\kappa\kappa$ is $\mathbf{\Sigma}^1_1$-definable if and only if it is definable over the structure $(\HH{\kappa^+},\in)$ by a $\Sigma_1$-formula with parameters. 
 Since $\HH{\kappa^+}=\LL_{\kappa^+}[x]$, these formulas are absolute between $\VV$ and $\HH{\kappa^+}$.
\end{proof}

\begin{corollary}
 Let $\kappa$ be an uncountable cardinal with $\kappa=\kappa^{{<}\kappa}$. Assume that $\VV=\LL[x]$ for some $x\subseteq\kappa$. Then every $\mathbf{\Sigma}^1_1$-subset of ${}^\kappa\kappa$ is equal 
 to an injective continuous image of a closed subset of ${}^\kappa\kappa$. \qed
\end{corollary}


In the other direction, we will show that it is also consistent that $C^{\kappa}$ is not a subclass of $I^{\kappa}_{cl}$.

\begin{theorem}\label{Theorem:ClubFilterNotInKappaSolovay}
 Let $\kappa$ be an uncountable cardinal with $\kappa=\kappa^{{<}\kappa}$ and let $G$ be either $\Add{\kappa}{\kappa^+}$-generic over $\VV$ or $\Coll{\kappa}{{<}\lambda}$-generic over $\VV$ 
 for some inaccessible cardinal $\lambda>\kappa$. In $\VV[G]$, the \emph{club filter} $\mathrm{Club}_\kappa$ is not equal to an injective continuous image of a closed subset of ${}^\kappa\kappa$.
\end{theorem}

We will also show that the classes $I^\kappa_{cl}$ and $S_1^{\LL,\kappa}$ do not coincide in any $Col(\kappa,{<}\lambda)$-generic extension, where $\lambda>\kappa$ is inaccessible.

\begin{theorem}\label{theorem:S1LKappaBaireInSolovay}
 Let $\kappa$ be an uncountable cardinal with $\kappa=\kappa^{{<}\kappa}$, $G$ be $\Coll{\kappa}{{<}\lambda}$-generic over $\VV$ for some inaccessible cardinal $\lambda>\kappa$ and 
 $M$ be an inner model of $\VV[G]$ with $M\subseteq\VV$. In $\VV[G]$, every set contained in the class $S_n^{M,\kappa}$ for some $n<\omega$ has the $\kappa$-Baire property. 
\end{theorem}

The following corollary is a direct consequence of the combination of the above result and Theorem \ref{non-separable I-kappa sets}.

\begin{corollary} \label{corollary: separating I-kappa from S-1} 
 Let $\kappa$ be an uncountable cardinal with $\kappa=\kappa^{{<}\kappa}$ and let $G$ be $\Coll{\kappa}{{<}\lambda}$-generic over $\VV$  for some inaccessible cardinal $\lambda>\kappa$.
 In $\VV[G]$, there is an element of the class $I^{\kappa}$ that is not contained in the class $S_n^{\LL,\kappa}$ for any $n<\omega$. \qed
\end{corollary}


\subsection{\texorpdfstring{The class $C^{\kappa,\mu}$ of continuous images of ${}^\kappa\mu$}{The class C-kappa-mu}}\label{subsection:ContImagesHigherCard}

We will consider images of continuous functions $\map{f}{{}^{\kappa}\mu}{{}^{\kappa}\kappa}$ for an arbitrary cardinal $\mu$.

Let us first consider the case where $\mu$ is a cardinal with $1<\mu<\kappa$. 
If $\kappa$ is not weakly compact, then the spaces ${}^{\kappa}\mu$ and ${}^{\kappa}\kappa$ are homeomorphic by the results of \cite{MR0367930}. Hence we may assume that $\kappa$ is weakly compact. 
Then the class $C^{\kappa,\mu}_{cl}$ of images of closed subsets of ${}^{\kappa}\mu$ under continuous functions $\map{f}{{}^{\kappa}\mu}{{}^{\kappa}\kappa}$ 
consists of closed subsets and its elements are exactly the sets of the form $[T]$, 
where $T$ is a subtree of ${}^{{<}\kappa}\kappa$ with the property that the $\alpha$-th level $T(\alpha)$ has cardinality less than $\kappa$ for every $\alpha<\kappa$. 
Theorem \ref{theorem:Main1} shows that the class $C^{\kappa,\mu}$ of images of ${}^{\kappa}\mu$ under continuous functions $\map{f}{{}^{\kappa}\mu}{{}^{\kappa}\kappa}$ is a proper subclass of $C^{\kappa,\mu}_{cl}$ in this case. 
The images of ${}^{\kappa}\mu$ under injective continuous functions $\map{f}{{}^{\kappa}\mu}{{}^{\kappa}\kappa}$ are exactly the sets of the form $[T]$, 
where $T$ is a perfect\footnote{A subtree $T$ of ${}^{{<}\kappa}\kappa\times{}^{{<}\kappa}\kappa$ is \emph{perfect} if it is ${<}\kappa$-closed and its splitting nodes are cofinal.}  subtree of ${}^{{<}\kappa}\kappa\times{}^{{<}\kappa}\kappa$ with the property that the level $T(\alpha)$ has size less than $\kappa$ for every   
$\alpha<\kappa$.

Next, we consider cardinals $\mu>\kappa$. We will later show (see Lemma \ref{lemma:WIsKappaPlusProjection}) that the closed set constructed in the proof of Theorem \ref{theorem:Main1} is equal to a continuous image of ${}^\kappa(\kappa^+)$. 
Since every subset of ${}^\kappa\kappa$ is obviously equal to a continuous image of ${}^\kappa(2^\kappa)$, this implies that 
\begin{equation*}
 c(\kappa) ~ = ~ \min\Set{\mu\in\On}{\textit{Every nonempty closed subset of ${}^\kappa\kappa$ is an element of $C^{\kappa,\mu}$}}
\end{equation*}
is a well-defined cardinal characteristic of every uncountable cardinal $\kappa$ satisfying $\kappa=\kappa^{{<}\kappa}$. 
By Theorem \ref{theorem:Main1} and the above remark, we have $\kappa<c(\kappa)\leq 2^\kappa$.  
In Section \ref{section:HigherCardTrees} and \ref{section:PerfectEmb} we will prove the following result that shows that we can manipulate the value of this cardinal characteristic by forcing.

\begin{theorem}\label{theorem:ManipulateCardinalCharC}
 Let $\kappa$ be an uncountable cardinal with $\kappa=\kappa^{{<}\kappa}$, $\mu\geq 2^\kappa$ be a cardinal with $\mu=\mu^\kappa$ and $\theta\geq\mu$ be a cardinal with $\theta=\theta^\kappa$. 
 Then the following statements hold in a cofinality preserving forcing extension $\VV[G]$ of the ground model $\VV$. 
 \begin{enumerate}
  \item $2^\kappa=\theta$. 

  \item Every closed subset of ${}^\kappa\mu$ is equal to a continuous image of ${}^\kappa\mu$. 

  \item There is a closed subset $A$ of ${}^\kappa\kappa$ that is not equal to a continuous image of ${}^\kappa\bar{\mu}$ for any $\bar{\mu}<\mu$ with $\bar{\mu}^{{<}\kappa}<\mu$.
 \end{enumerate}
\end{theorem}

In particular, this result shows that the class of nonempty sets in $C^{\kappa,\mu}$ and the class $C^{\kappa,\mu}_{cl}$ can coincide for some $\kappa<\mu<2^\kappa$ (see Lemma \ref{lemma:AddForcesMuKappa}), 
but this statement does not follow from the axioms of $\ZFC$. The following result will be one of the key ingredients in the proof of Theorem \ref{theorem:ManipulateCardinalCharC}.

\begin{theorem}\label{theorem:TreesFromCohenReal}
 Let $c$ be $\Add{\omega}{1}$-generic over $\VV$. In $\VV[c]$, if $\kappa$ is an uncountable regular cardinal, then there is a closed subset $A$ of ${}^\kappa\kappa$  
 such that $A$ is not a continuous image of ${}^\kappa\mu$ for every cardinal $\mu$ with $\mu^{{<}\kappa}<2^\kappa$. 
\end{theorem}

Given a cardinal $\mu$, we next consider the class $C^{\kappa,\mu}_{cl}$ consisting of all continuous images of closed subsets of ${}^\kappa\mu$. This class contains all $\mathbf{\Sigma}^1_1$-subsets. 
We will discuss results showing that these classes can consistently be equal.  
Let $\goedel{\alpha}{\beta}$ denote the G\"odel pair of $\alpha,\beta\in\On$.  
Given $x\in{}^\kappa\mu$ and $\alpha<\kappa$, we let $(x)_\alpha$ denote the element of ${}^\kappa\mu$ defined by $(x)_\alpha(\beta)=x(\goedel{\alpha}{\beta})$ for all $\beta<\kappa$.

\begin{proposition}\label{proposition:BigTreesDefinable}
 Let $\kappa$ be an uncountable regular cardinal and $\mu$ be a cardinal with $\mu=\mu^{{<}\kappa}<2^\kappa$. Assume that every subset of ${}^\kappa\kappa$ of cardinality $\mu$ is a $\mathbf{\Sigma}^1_1$-subset. 
 Then every set in $C^{\kappa,\mu}_{cl}$ is a $\mathbf{\Sigma}^1_1$-subset. 
\end{proposition}

\begin{proof}
 Let $A$ be a closed subset of ${}^\kappa\mu$ and $\map{f}{A}{{}^\kappa\kappa}$ be a continuous function. Then the graph of $f$ is a closed subset of ${}^\kappa\mu\times{}^\kappa\kappa$ 
 and $\ran{f}=p[T]$ for some subtree $T$ of ${}^{{<}\kappa}\kappa\times{}^{{<}\kappa}\mu$. Fix a family $\seq{y_t}{t\in{}^{{<}\kappa}\mu}$ of pairwise distinct elements of ${}^\kappa\kappa$. 
 Our assumption implies that the sets $$B ~ = ~ \Set{(x,y_t)\in{}^\kappa\kappa\times{}^\kappa\kappa}{t\in{}^{{<}\kappa}\mu, ~ (x\restriction\length{t},t)\in T, ~ \supp{x}\subseteq\length{t}}$$ and 
 $C=\Set{(y_s,y_t)}{s,t\in{}^{{<}\kappa}\mu, ~ s\subsetneq t}$ are $\mathbf{\Sigma}^1_1$-subsets of ${}^\kappa\kappa$. Then 
 \begin{equation*}
  \begin{split}
   x\in\ran{f} ~ \Longleftrightarrow ~ \exists y ~ [ & \forall\alpha<\beta<\kappa ~ ((y)_\alpha,(y)_\beta)\in C \\
                                                    &  ~ \wedge ~ \forall \alpha<\kappa ~  \exists(\bar{x},\bar{y})\in B ~ [x\restriction\alpha=\bar{x}\restriction\alpha\wedge(y)_\alpha=\bar{y}] ]
  \end{split}  
 \end{equation*}
 and this equivalence shows that $\ran{f}$ is definable in $(\HH{\kappa^+},\in)$ by a $\Sigma_1$-formula with parameters. By the above remarks, this shows that $\ran{f}$ is a $\mathbf{\Sigma}^1_1$-subset of ${}^\kappa\kappa$.
\end{proof}

We show that the above assumption holds in the canonical forcing extension of $\LL$ that is a model of \emph{Baumgartner's axiom $\rm{BA}(\kappa)$},  i.e. that is a model of the statement that for every 
$\kappa$-linked, ${<}\kappa$-closed, well-met partial order $\PPP$ and every collection $\calD$ of $\kappa^+$-many dense subsets of $\PPP$, there is a $\calD$-generic filter for $\PPP$ (see {\cite[Section 4]{MR823775}} and \cite{MR1278025}).  
Note that the canonical partial order that forces $\rm{BA}(\kappa)$ for some uncountable cardinal $\kappa$ with $\kappa=\kappa^{{<}\kappa}$ is ${<}\kappa$-closed and satisfies the ${\kappa}^+$-chain condition (see {\cite[Theorem 4.2]{MR823775}} and {\cite[Theorem 0.3]{MR1278025}}). 
The following proof is a direct generalization of the arguments of {{\cite[Section 3.2]{MR0270904}}} to higher cardinalities.

\begin{lemma}
 Let $\kappa$ be an uncountable cardinal with $\kappa=\kappa^{{<}\kappa}$ and $\kappa^+=(\kappa^+)^\LL$. Assume that  $\rm{BA}(\kappa)$ holds. 
 Then every subset of ${}^\kappa\kappa$ of cardinality $\kappa^+$ is contained in the class $S_1^{\LL,\kappa}$. In particular, every such set is a $\mathbf{\Sigma}^1_1$-subset. 
\end{lemma}

\begin{proof}
 Fix a subset $A=\Set{y_\gamma}{\gamma<\kappa^+}$ of ${}^\kappa\kappa$ of cardinality $\kappa^+$. Then there is a set $W=\Set{x_\gamma}{\gamma<\kappa^+}\in\LL$ such that 
 $$(x_\gamma)_\alpha = (x_\delta)_\beta ~ \Longleftrightarrow ~ \alpha=\beta ~ \wedge ~ \gamma=\delta$$ holds for all $\alpha,\beta<\kappa$ and $\gamma,\delta<\kappa^+$.  
 Define $$B ~ = ~ \Set{(x_\gamma)_{\goedel{\alpha}{y_\gamma(\alpha)}}}{\alpha<\kappa,~\gamma<\kappa^+} ~ \subseteq ~ \LL.$$

 Let $\QQQ(B)$ denote the corresponding generalization of the \emph{almost disjoint coding forcing} to cardinality $\kappa$ (see, for example, {\cite[Section 4]{MR2987148}}). Since $\QQQ(B)$ is  $\kappa$-linked, ${<}\kappa$-closed 
 and well-met, the axiom $\rm{BA}(\kappa)$ implies that there is an enumeration $\seq{s_\alpha}{\alpha<\kappa}$ and a function $c\in{}^\kappa 2$ such that the 
 equation $$z\in B ~ \Longleftrightarrow ~ \exists\alpha<\kappa ~ \forall \alpha\leq\beta<\kappa ~ [s_\beta\subseteq z \longrightarrow c(\beta)=1]$$ holds for every $z\in({}^\kappa\kappa)^\LL$.  
 Since  $$y\in A ~ \Longleftrightarrow ~ \exists x\in B ~ \forall\alpha,\beta<\kappa ~ [y(\alpha)=\beta\longleftrightarrow (x)_{\goedel{\alpha}{\beta}}\in B]$$
 holds for every $y\in{}^\kappa\kappa$ and 
$B$ is a subset of $({}^\kappa\kappa)^\LL$, the parameter $c$ witnesses that the set $A$ is contained in the class $S_1^{\LL,\kappa}$.   
\end{proof}

\begin{corollary}
 Let $\kappa$ be an uncountable cardinal with $\kappa=\kappa^{{<}\kappa}$ and $\kappa^+=(\kappa^+)^\LL$. If $\rm{BA}(\kappa)$ holds, 
 then the classes $C^{\kappa,\kappa^+}_{cl}$ and $\mathbf{\Sigma}^1_1(\kappa)$ are equal. \qed
\end{corollary}

Conversely, it is also consistent that the $\GCH$ fails at $\kappa$ and $C^{\kappa,\kappa^+}_{cl}$ is not a subclass of $\mathbf{\Sigma}^1_1(\kappa)$.

\begin{proposition}
 Let $\kappa$ be an uncountable regular cardinal, $\theta>\kappa$ be an inaccessible cardinal and $G*H$ be $(\Coll{\kappa}{{<}\theta}*\Add{\check{\kappa}}{\check{\kappa}^{++})}$-generic over $\VV$. 
 In $\VV[G,H]$, there is a continuous image of ${}^\kappa(\kappa^+)$ that is not a $\mathbf{\Sigma}^1_1$-subset of ${}^\kappa\kappa$.
\end{proposition}

\begin{proof}
 By {\cite[Proposition 9.9]{MR2987148}}, every $\mathbf{\Sigma}^1_1$-subset of ${}^\kappa\kappa$ in $\VV[G,H]$ of cardinality greater than $\kappa$ has cardinality at least $\kappa^{++}$. 
 Since every subset of ${}^\kappa\kappa$ of cardinality $\kappa^+$ is a continuous image of ${}^\kappa(\kappa^+)$, the above statement follows directly. 
\end{proof}


 \section{Thin sets}\label{Section:thin}

From now on, unless otherwise noted, we let $\kappa$ denote an uncountable cardinal with $\kappa=\kappa^{{<}\kappa}$. 
The goal of this section is to show that there is a closed subset of ${}^{\kappa}\kappa$ that is not contained in the class $C^{\kappa}$.

\begin{definition} 
 Let $\mu\leq\kappa$ be a cardinal. A set $A\subseteq {}^{\kappa}\kappa$ is \emph{$\mu$-thin} if $A\neq p[T]$ for every ${<}\mu$-closed subtree $T$ of ${}^{{<}\kappa}\kappa\times{}^{{<}\kappa}\kappa$ without end nodes. 
\end{definition} 
 
We will construct a $\kappa$-thin closed subset of ${}^{\kappa}\kappa$. In the following, we call a subset $D$ of ${}^{{<}\kappa}\mu_0\times{}^{{<}\kappa}\mu_1$ a \emph{${<}\kappa$-closed subset of the tree ${}^{{<}\kappa}\mu_0\times{}^{{<}\kappa}\mu_1$} 
if $D$ consists of pairs of functions of equal length, and $D$ is closed under increasing unions of length $\gamma$ for all $\gamma<\kappa$. 
Given such a subset $D$, we let $[D]$ denote the set of all 
$(x,y)\in {}^\kappa\mu_0\times{}^\kappa\mu_1$ such that the set 
$\Set{\alpha<\kappa}{(x\restriction\alpha, ~ y\restriction\alpha)\in D}$ is unbounded in $\kappa$ and 
define $p[D]$ to be the projection of $[D]$ onto the first coordinate.

\begin{lemma}\label{continuous image to projection}
 Let $\mu$ and $\lambda$ be cardinals with $\mu=\mu^{{<}\kappa}$ and $\lambda=\lambda^{{<}\kappa}$. Suppose that $\map{f}{{}^{\kappa}\lambda}{{}^{\kappa}\mu}$ is continuous. 
 Then there is a ${<}\kappa$-closed subset $D$ of the tree ${}^{{<}\kappa}\mu\times{}^{{<}\kappa}\lambda$ without end nodes such that $p[D]=\ran{f}$. 
\end{lemma} 
 
\begin{proof}
 We define $$D ~ = ~ \Set{(s,t)\in {}^{\gamma}\lambda\times {}^{\gamma}\mu}{\gamma<\kappa,\ f[N_s]\subseteq N_t}.$$ Note that $D$ is ${<}\kappa$-closed. 
 
 \begin{claim*} 
  For all $x\in {}^{\kappa}\lambda$, $s_0\subsetneq x$, and $t_0\subsetneq f(x)$, there is a pair $(s,t)$ in $D$ with $s_0\subseteq s\subseteq x$ and $t_0\subseteq  t\subseteq f(x)$. 
 \end{claim*} 
 
 \begin{proof}[Proof of the Claim]
  We construct strictly increasing sequences $\seq{s_n\in{}^{<\kappa}\lambda}{n<\omega}$ and $\seq{t_n\in {}^{<\kappa}\mu}{n<\omega}$ with $s_n\subseteq x$, $t_n\subseteq f(x)$, $f[N_{s_n}]\subseteq N_{t_n}$ and  $$\length{t_{n+1}} ~\geq ~ \length{s_{n+1}} ~ \geq ~ \length{t_n}$$ for all $n<\omega$,  
  using the continuity of $f$. Let $(s,t)=(\bigcup_{n\in\omega} s_n, \bigcup_{n\in\omega} t_n)\in D$. 
 \end{proof}

 It follows that $D$ has no end nodes and that $(x,f(x))\in [D]$ for all $x\in {}^{\kappa}\mu$.

 \begin{claim*} 
  $[D]=\{(x,f(x))\mid x\in {}^{\kappa}\lambda\}$. 
 \end{claim*} 
 
 \begin{proof}[Proof of the Claim]
  Suppose that $(x,y)\in [D]$ and that $f(x)\neq y$. Let $t\in {}^{<\kappa}\lambda$ with $y\in N_t$ and $f(x)\notin N_t$. Find $(u,v)\in D$ with $u\subseteq x$, $v\subseteq y$, and $\length{v}\geq\length{t}$. 
  Then $t\subseteq  v$ and $f(x)\in f[N_u]\subseteq N_v\subseteq N_t$, contradicting the assumption that $f(x)\notin N_t$.  
 \end{proof}

 This shows $p[D]=\ran{f}$, completing the proof of the lemma. 
 \end{proof}

 Given a set $D$ as above, we want to find a ${<}\kappa$-closed tree $T$ without end nodes such that $p[T]=p[D]$. Note that the downwards-closure of a ${<}\kappa$-closed subset $D$ does not necessarily have the same projection as the  subset itself. For example, if we consider the set $$D ~ = ~ \Set{(s,t)\in{}^\gamma 2\times{}^\gamma 2}{\gamma<\kappa, ~ \exists\alpha<\gamma ~ s(\alpha)=1},$$ then the function with domain $\kappa$ and constant value $0$ is not an element of $p[D]$, but it is contained in the projection of the downwards-closure of $D$.

 \begin{lemma}\label{projection to projection of tree} 
  Suppose that $\lambda\leq \mu$ are cardinals with  $\mu=\mu^{{<}\kappa}$. Suppose that $D$ is a ${<}\kappa$-closed subset of the tree ${}^{{<}\kappa}\lambda\times{}^{{<}\kappa}\mu$ without end nodes. 
  Then there is a ${<}\kappa$-closed subtree $T$ of ${}^{{<}\kappa}\lambda\times{}^{{<}\kappa}\mu$ without end nodes such that $p[D]=p[T]$. 
 \end{lemma} 
 
\begin{proof} 
 Let $\seq{d_{\alpha}}{\alpha<\mu}$ be an (not necessarily injective) enumeration of $D$. Let $\bar{T}$ denote the subtree of ${}^{{<}\kappa}\lambda\times{}^{{<}\kappa}\mu\times{}^{{<}\kappa}\mu$ 
 consisting of all $(s,t,u)\in {}^\gamma\lambda\times {}^\gamma\mu\times{}^\gamma\mu$  such that $\gamma<\kappa$ and the following statements hold. 
 \begin{itemize} 
 \item $\seq{d_{u(\alpha)}}{\alpha<\gamma}$ is a weakly increasing sequence in $D$. 
 \item $(s\restriction\alpha, ~ t\restriction\alpha)$ is a node equal to or below $u(\alpha)$ in ${}^{{<}\kappa}\lambda\times{}^{{<}\kappa}\mu$ for all $\alpha<\gamma$. 
 \end{itemize}

 \begin{claim*} 
  $\bar{T}$ is ${<}\kappa$-closed. \qed
 \end{claim*}

 \begin{claim*} $\bar{T}$ has no end nodes. 
 \end{claim*} 
 
 \begin{proof}[Proof of the Claim]
  Suppose that $(s,t,u)\in\bar{T}$ has length $\gamma<\kappa$.  
  Since $D$ is ${<}\kappa$-closed and has no end nodes, there is a $\beta<\mu$ such that $d_\beta$ has length greater than $\gamma$ and $d_{u(\alpha)}$ is below $d_\beta$ 
  in ${}^{{<}\kappa}\lambda\times{}^{{<}\kappa}\mu$ for every $\alpha<\gamma$. 
  Let $d_\beta=(s_*,t_*)$ and define $\map{u_*}{\gamma+1}{\kappa}$ by $u_*\restriction\gamma=u\restriction\gamma$ and $u_*(\gamma)=\beta$. 
  Then the tuple $$(s_*\restriction(\gamma+1), ~ t_*\restriction(\gamma+1), ~ u_*)$$ is a direct successor of $(s,t,u)$ in $\bar{T}$.  
 \end{proof}

Since $p[D]=p(p[\bar{T}])$, there is a ${<}\kappa$-closed subtree $T$ of ${}^{{<}\kappa}\lambda\times{}^{{<}\kappa}\mu$ without end nodes such that $p[T]=p(p[\bar{T}])=p[D]$.
 \end{proof}

In the following, we construct a $\kappa$-thin closed subset of ${}^{\kappa}\kappa$. 
Given $\gamma\leq\kappa$ closed under G\"odel pairing and $\map{s}{\gamma}{2}$, we define 
\begin{equation*}
 R_s ~ = ~ \Set{(\alpha,\beta)\in\gamma\times\gamma}{s(\goedel{\alpha}{\beta})=1}.
\end{equation*} 
 Set  
\begin{equation}\label{equation:DefinitionOfW}
 W ~ = ~ \Set{x\in {}^{\kappa}2}{\textit{$(\kappa,R_x)$ is a well-order}}
\end{equation}
Then $W$ is a closed subset of ${}^\kappa\kappa$ with $W=[S]$, where $S=\Set{x\restriction\alpha}{x\in W,~\alpha<\kappa}$. 
Given $x\in W$ and $\alpha<\kappa$, let $\rank{\alpha}{x}$ denote the rank of $\alpha$ in $(\kappa,R_x)$. In the following, we write $\alpha<_x\beta$ instead of $(\alpha,\beta)\in R_x$.

\begin{lemma}\label{lemma:WisNotProjectionClosedTree}
 There is no $\omega$-closed tree subtree $T$ of ${}^{<\kappa}\kappa\times {}^{<\kappa}\kappa$ with cofinal branches through all its nodes and $W=p[T]$. 
\end{lemma} 
 
\begin{proof} 
 Assume, towards a contradiction, that $T$ is such a tree.  
 Given $(s,t)\in T$ and $\alpha<\kappa^+$, we define 
 \begin{equation*}
  T_{s,t} ~ = ~ \Set{(u,v)\in T}{(s\subseteq u \vee u\subseteq s)\wedge (t\subseteq v\vee v\subseteq t)}
 \end{equation*}
 to be subtree of $T$ induced by the node $(s,t)$ and
 \begin{equation*}
  r(s,t,\alpha) ~ = ~ \sup\Set{\rank{\alpha}{x}}{x\in p[T_{s,t}]} ~ \leq ~ \kappa^+  
 \end{equation*}
 to be the supremum of the ranks of $\alpha$ in $(\kappa,R_x)$ with $x\in p[T_{s,t}]$. 
 Then we have $r(\emptyset,\emptyset,\alpha)=\kappa^+$ for all $\alpha<\kappa$.

 \begin{claim*} 
  Let $(s,t)\in T$ and $\alpha<\kappa$ with $r(s,t,\alpha)=\kappa^+$. 
  If $\gamma<\kappa^+$, then there is $(u,v)\in T$ extending $(s,t)$ and $\beta<\kappa$ such that $\dom{u}$ is closed under G\"odel pairing, $\alpha<\beta<\length{u}$, $\beta<_u\alpha$, and $r(u,v,\beta)\geq\gamma$. 
 \end{claim*} 
 
 \begin{proof}[Proof of the Claim]
  The assumption $r(s,t,\alpha)=\kappa^+$ allows us to find $(x,y)\in [T]$ with $s\subseteq x$, $t\subseteq y$ and $\rank{\alpha}{x}\geq\gamma+\kappa$. 
  This implies that there is a $\beta$ with $\alpha<\beta<\kappa$ and $\gamma\leq\rank{\beta}{x}<\gamma+\kappa$. Pick $\delta>\max\{\alpha,\beta,\length{s}\}$ closed under G\"odel pairing and define 
  $(u,v)$ to be the node $(x\restriction\delta, y\restriction\delta)$ extending $(s,t)$. Since $R_u$ is a well-ordering of $\length{u}$, we have $\beta<_u\alpha$. Finally, $(x,y)$ witnesses that $r(u,v,\beta)\geq\gamma$. 
 \end{proof}

 \begin{claim*} 
  If $(s,t)\in T$ and $\alpha<\kappa$ with $r(s,t,\alpha)=\kappa^+$, then there is a node $(u,v)$ in $T$ extending $(s,t)$ and $\alpha<\beta<\length{u}$ such that $\length{u}$ is closed under G\"odel pairing, $\beta<_u \alpha$ and $r(u,v,\beta)=\kappa^+$. 
 \end{claim*} 
 
 \begin{proof}[Proof of the Claim]
  Given $\gamma<\kappa^+$, let $(u_\gamma,v_\gamma)\in T$ and $\beta_\gamma<\kappa$ denote the objects obtained by an application of the previous claim.
  Then we can find $(u,v)\in T$, $\beta<\kappa$ and $X\in[\kappa^+]^{\kappa^+}$ such that $(u,v)=(u_\gamma,v_\gamma)$ and $\beta=\beta_\gamma$ for every $\gamma\in X$. This implies that $r(u,v,\beta)=\kappa^+$. 
 \end{proof}

 These claims imply that there are strictly increasing sequences 
 $\seq{(s_n, t_n)}{n<\omega}$ of nodes in $T$ and $\seq{\beta_n}{n<\omega}$ of elements of $\kappa$ such that 
 $\length{s_n}$ is closed under G\"odel pairing and $\beta_{n+1} <_{s_{n+1}} \beta_n$ for all $n<\omega$.  Let $s=\bigcup_{n<\omega} s_n$ and $t=\bigcup_{n<\omega} t_n$. 
 Then $\length{s}$ is closed under G\"odel pairing and $R_s$ is ill-founded, hence $s\notin S$. But $(s,t)\in T$, since $T$ is $\omega$-closed. 
 By our assumptions on $T$, there is a cofinal branch $(x,y)$ in $[T]$ through $(s,t)$ and this implies that $s=x\restriction\length{s}\in S$, a contradiction. 
 \end{proof}

 In particular, the closed subset $W$ is not $\kappa$-thin.

 \begin{proof}[Proof of Theorem \ref{theorem:Main1}]
  Let $\kappa$ be an uncountable cardinal with $\kappa=\kappa^{{<}\kappa}$ and let $W$ be the closed subset of ${}^\kappa\kappa$ defined in (\ref{equation:DefinitionOfW}). 
  Assume, towards a contradiction, that there is a continuous function $\map{f}{{}^\kappa\kappa}{{}^\kappa\kappa}$ with $W=\ran{f}$. 
  By Lemma \ref{continuous image to projection} and \ref{projection to projection of tree}, there is a ${<}\kappa$-closed subtree $T$ of ${}^{{<}\kappa}\kappa\times{}^{{<}\kappa}\kappa$ without end nodes such that $W=p[T]$. 
  This contradicts Lemma \ref{lemma:WisNotProjectionClosedTree}. 
 \end{proof}

We cannot omit the assumption on the existence of cofinal branches through $T$ in Lemma \ref{lemma:WisNotProjectionClosedTree} by the following remark.

\begin{proposition} 
 Suppose that $\mu<\kappa$. Then every closed subset of ${}^{\kappa}\kappa$ is the projection of a ${<}\mu$-closed tree without end nodes. 
\end{proposition} 

\begin{proof} 
 Suppose that $S\subseteq{}^{<\kappa}\kappa$ is a tree. We add a fully branching tree of height $\mu$ at each node $t$ in the boundary $\partial S$ of $S$ (see (\ref{equation:DefBorderT})), 
 i.e. we consider the tree $$T ~ = ~ S ~ \cup ~ \Set{t\in{}^{{<}\kappa}\kappa}{\exists \alpha\leq\length{t} ~[t\restriction\alpha\in\partial S ~ \wedge ~ \length{t}<\alpha+\mu]}.$$  
 Then $T$ is a ${<}\mu$-closed tree without end nodes and $[S]=[T]$. 
\end{proof}


\section{\texorpdfstring{Injective images of ${}^{\kappa}\kappa$}{Injective images of kappa-kappa}}\label{section:I kappa}

We now turn to the class $I^{\kappa}$ of images of ${}^{\kappa}\kappa$ under continuous injective maps. Note that there are $\kappa$-Borel subset of ${}^{\kappa}\kappa$ that are not contained in $I^{\kappa}$, 
because every set in $I^{\kappa}$ has no isolated points.

Before we prove the more general results about this class mentioned in the Introduction (Theorem \ref{non-separable I-kappa sets} and Corollary \ref{non-Borel I-kappa sets}), 
we motivate their proofs by showing that, under certain cardinal arithmetic assumptions on $\kappa$, well-known combinatorial objects provide examples of sets in $I^\kappa$  
that do not have the $\kappa$-Baire property and therefore are not $\kappa$-Borel subsets of ${}^\kappa\kappa$. This set will consists of the \emph{trivial $\nu$-coherent sequences $C$-sequences} in the case where $\kappa=\nu^+=2^\nu$ with $\nu$ regular.

\begin{definition} 
 Let $\nu$ be an infinite regular cardinal. 
 \begin{enumerate} 
  \item Given $\gamma\in\On$, a sequence $\seq{C_\alpha}{\alpha<\gamma}$ is a \emph{$C$-sequence} if the following statements hold for all $\alpha<\gamma$. 
   \begin{itemize}
    \item If $\alpha$ is a limit ordinal, then $C_\alpha$ is a closed unbounded subset of $\alpha$. 

    \item If $\alpha=\bar{\alpha}+1$, then $C_\alpha=\{\bar{\alpha}\}$. 
   \end{itemize}

 \item A $C$-sequence $\seq{C_\alpha}{\alpha<\gamma}$ with $\gamma\in\Lim$ is \emph{trivial} if there is a closed unbounded subset $C$ of $\gamma$ such that for every $\alpha<\gamma$ 
  there is a $\alpha\leq\beta<\gamma$ with $C\cap\alpha\subseteq C_\beta$. 

 \item A $C$-sequence $\seq{C_\alpha}{\alpha<\gamma}$ is \emph{$\nu$-coherent} if $C_\alpha=C_\beta\cap\alpha$ for all $\beta<\gamma$ and $\alpha\in\Lim(C_\beta)$ with $\cof{\alpha}=\nu$. 

 \item Given a $\nu$-coherent $C$-sequence $\vec{C}=\seq{C_\alpha}{\alpha<\gamma}$ with $\gamma\in\Lim$, we say that a closed unbounded subset $C_\gamma$ of $\gamma$ is a \emph{$\nu$-thread through $\vec{C}$} if  
  the sequence $\seq{C_\alpha}{\alpha\leq\gamma}$ is also a $\nu$-coherent $C$-sequence. 
 \end{enumerate} 
\end{definition}

It is easy to see that $\nu$-coherent $C$-sequences are also coherent with respect to every regular cardinal greater than $\nu$. 
This shows that $\omega$-coherent $C$-sequences 
are \emph{square sequences} (in the sense of {\cite[Definition 7.1.1]{MR2355670}}).

\begin{proposition}\label{proposition:TrivialMuCoherentSequences}
 Let $\nu$ is an infinite regular cardinal, $\lambda$ be a limit ordinal with $\cof{\lambda}>\nu$ and $\vec{C}=\seq{C_\alpha}{\alpha<\lambda}$ be a $\nu$-coherent $C$-sequence.
 Then $\vec{C}$ is trivial if and only if there is a $\nu$-thread through $\vec{C}$. 
\end{proposition}

\begin{proof}
 Assume that $C\subseteq\gamma$ witnesses that $\vec{C}$ is trivial. Let $D$ denote the set of all $\alpha\in\Lim(C)$ with $\cof{\alpha}=\nu$. 
 Fix $\alpha,\bar{\alpha}\in D$ with $\bar{\alpha}<\alpha$. Then there is a $\alpha\leq \beta<\gamma$ such that $C\cap\alpha\subseteq C_\beta$. 
 This shows that $\alpha$ is a limit point of $C_\beta$ and $\nu$-coherency implies $C\cap\alpha\subseteq C_\beta\cap\alpha=C_\alpha$. 
 We can conclude that $\bar{\alpha}\in\Lim(C_\alpha)$ and $C_{\bar{\alpha}}=C_\alpha\cap\bar{\alpha}$.  
 Define $C_\lambda=\bigcup\Set{C_\alpha}{\alpha\in D}$. By the above computations, we have $C_\lambda\cap\alpha=C_\alpha$ for all $\alpha\in D$ and $\cof{\lambda}>\nu$ implies that $C_\lambda$ is a $\nu$-thread through $\vec{C}$.

 Now, let $C_\lambda$ be a $\nu$-thread through $\vec{C}$. Given $\alpha<\lambda$, the assumption $\cof{\lambda}>\nu$ shows that there is a $\beta\in\Lim(C_\lambda)\setminus\alpha$ with $\cof{\beta}=\nu$ 
 and in this situation, $\nu$-coherency implies $C_\lambda\cap\alpha=C_\beta\cap\alpha\subseteq C_\beta$. This shows that $C_\lambda$ witnesses that $\vec{C}$ is trivial. 
\end{proof}

In the following, we assume that $\kappa=\nu^+=2^\nu$ for some regular cardinal $\nu$. 
Let $\mathscr{Coh}(\kappa,\nu)$ denote the set of all $\nu$-coherent $C$-sequences of length $\kappa$. 
We equip $\mathscr{Coh}(\kappa,\nu)$ with the topology whose basic open sets consist of all extensions of 
$\nu$-coherent $C$-sequences of length less than $\kappa$.

\begin{proposition}
 The space $\mathscr{Coh}(\kappa,\nu)$ is homeomorphic to ${}^{\kappa}\kappa$. 
\end{proposition} 

\begin{proof} 
We will represent $\mathscr{Coh}(\kappa,\nu)$ as the set of branches of a tree $T$ isomorphic to ${}^{<\kappa}\kappa$. 
Define $T$ to be the tree consisting of all $\nu$-coherent $C$-sequences $\seq{C_\alpha}{\alpha<\gamma}$ such that 
either $\gamma<\kappa$ is a limit of ordinals of cofinality $\nu$ or $\gamma=\bar{\gamma}+1<\kappa$ 
and $\bar{\gamma}$ is a limit ordinal of cofinality $\nu$ that is not a limit of ordinals of cofinality $\nu$. 
We first show that $T$ is $\kappa$-branching.

 \begin{claim*}
  Suppose that $\vec{C}$ is a $\nu$-coherent $C$-sequence of length $\gamma<\kappa$ and that $\delta$ is the least ordinal of cofinality $\nu$ greater than $\gamma$. 
  Then there are $\kappa$-many sequences in $T$ of length $\delta+1$ that extend $\vec{C}$. 
 \end{claim*}

 \begin{proof}[Proof of the Claim]
  Since there are $\kappa$-many closed unbounded subsets of $\delta$ of order-type $\nu$, there are also $\kappa$-many sequences $\seq{C_\alpha}{\alpha\in\Lim\cap[\gamma,\delta]}$ 
  with the property that $C_\alpha$ is a closed unbounded subset of $\alpha$ of order-type at most $\nu$ for every limit ordinal in the interval $[\gamma,\delta]$. 
  Since every such sequence determines a unique direct successor of $\vec{C}$ in $T$, this completes the proof of the proposition. 
 \end{proof}

  Since the tree $T$ is closed under increasing unions of size ${<}\kappa$, the above claim implies that 
  the trees $T$ and ${}^{{<}\kappa}\kappa$ are isomorphic.  
  This yields the statement of the proposition, because the collection of all subsets of $\mathscr{Coh}(\kappa,\nu)$ consisting of sets of all extensions of a given element of $T$ forms a basis of the topology on this space. 
\end{proof}

In the following, we will always identify the space $\mathscr{Coh}(\kappa,\nu)$ with ${}^\kappa\kappa$.

\begin{theorem}\label{theorem:THREADSEQNOTBaire}
 Suppose that $\nu$ is an infinite regular cardinal and $\kappa=\nu^+=2^\nu$. Let $\mathscr{Triv}(\kappa,\nu)$ denote the subset of $\mathscr{Coh}(\kappa,\nu)$ consisting of all trivial $\nu$-coherent $C$-sequences.
 \begin{enumerate}
  \item $\mathscr{Triv}(\kappa,\nu)$ is a continuous injective image of ${}^\kappa\kappa$. 

  \item $\mathscr{Triv}(\kappa,\nu)$ does not have the $\kappa$-Baire property. 
 \end{enumerate}
\end{theorem} 

\begin{proof} 
 Define $\mathscr{Thr}(\kappa,\nu)$ to be the set of all pairs $(\vec{C},C)$ such that $\vec{C}$ is an element of $\mathscr{Coh}(\kappa,\nu)$ and $C$ is a $\nu$-thread through $\vec{C}$. 
 Then $\mathscr{Triv}(\kappa,\nu)=p[\mathscr{Thr}(\kappa,\nu)]$, where $\map{p}{\mathscr{Thr}(\kappa,\nu)}{\mathscr{Coh}(\kappa,\nu)}$ is the projection onto the first coordinate.

 Let $\calS$ be the tree consisting of all pairs $(\vec{D},D)$ such that  $\vec{D}$ is a $\nu$-coherent $C$-sequence of length $\gamma\in\Lim\cap\kappa$ and $D$ is a thread through $\vec{D}$. 
 Given $(\vec{D},D)$ in $\calS$, the set $D$ is unbounded in the length of $\vec{D}$ and this allows us to identify $\mathscr{Thr}(\kappa,\nu)$ with the set of all cofinal branches through $\calS$ 
 and equip $\mathscr{Thr}(\kappa,\nu)$ with the topology whose basic open sets consist of all extensions of elements of $\calS$.

\begin{claim*}
 The space $\mathscr{Thr}(\kappa,\nu)$ is homeomorphic to ${}^{\kappa}\kappa$.
\end{claim*}

\begin{proof}[Proof of the Claim]
 By the above remarks, it again suffices to show that the tree $\calS$ is isomorphic to ${}^{<\kappa}\kappa$. 
 Since $\calS$ is also ${<}\kappa$-closed, we have to show that every node in $\calS$ has $\kappa$-many direct successors. Fix such a pair $(\seq{D_\alpha}{\alpha<\gamma},D_\gamma)\in\calS$. 
 Then the sequence $\vec{D}_*=\seq{D_\alpha}{\alpha\leq\gamma}$ is also a $\nu$-coherent $C$-sequence. Given an ordinal $\gamma<\delta<\kappa$ of countable cofinality, there is a $\nu$-coherent $C$-sequence $\vec{D}_\delta=\seq{D^\delta_\alpha}{\alpha<\delta}$ 
 extending $\vec{D}_*$ such that $\otp{D^\delta_\alpha}\leq\nu$ for every $\gamma<\alpha<\delta$. If $D^\delta$ is a closed unbounded subset of $\delta$ with $\gamma\in D^\delta\subseteq\delta\setminus\gamma$ and $\otp{D^{\delta}}=\omega$, then 
 the pair $(\vec{D}_\delta,D_\gamma\cup D^\delta)$ is a direct successor of $\vec{D}$ in $\calS$. Since there are $\kappa$-many ordinals $\delta$ with these properties, this completes the proof of the claim. 
\end{proof}

We now show that $\mathscr{Triv}(\kappa,\nu)$ is an element of $I^{\kappa}$.

\begin{claim*} 
 The projection $\map{p}{\mathscr{Thr}(\kappa,\nu)}{\mathscr{Coh}(\kappa,\nu)}$ is injective. 
\end{claim*} 

\begin{proof}[Proof of the Claim]
 Suppose there is a sequence $\vec{C}$ in $\mathscr{Coh}(\kappa,\nu)$ and closed unbounded subsets $C$ and $D$ of $\kappa$ such that $(\vec{C},C),(\vec{C},D)\in\mathscr{Thr}(\kappa,\nu)$. 
 Then there are arbitrarily large $\alpha<\kappa$ with $\mathrm{cof}(\alpha)=\nu$ and $C\cap \alpha= C_{\alpha}=D\cap \alpha$. This implies $C=D$. 
\end{proof}

\begin{claim*} 
 The projection $\map{p}{\mathscr{Thr}(\kappa,\nu)}{\mathscr{Coh}(\kappa,\nu)}$ is continuous. 
\end{claim*} 

\begin{proof}[Proof of the Claim]
 Let $(\vec{C},C)\in\mathscr{Thr}(\kappa,\nu)$ and $\alpha<\kappa$. Let $\gamma$ be a limit point of $C$ above $\alpha$. Then the pair $(\vec{C}\restriction\gamma,C\cap\gamma)$ is an element of $\calS$ 
 and the projection of every extension of this pair in $\mathscr{Thr}(\kappa,\nu)$ is an extension of $\vec{C}\restriction\alpha$ in $\mathscr{Coh}(\kappa,\nu)$. 
\end{proof}

It remains to show that the set $\mathscr{Triv}(\kappa,\nu)$ does not have the $\kappa$-Baire property.

\begin{claim*}
 The set $\mathscr{Triv}(\kappa,\nu)$ does not have the $\kappa$-Baire property in $\mathscr{Coh}(\kappa,\nu)$. 
\end{claim*} 

\begin{proof} 
 We first argue that $\mathscr{Triv}(\kappa,\nu)$ is not comeager in any non-empty open set, i.e. if $U$ is an nonempty open subset of $\mathscr{Coh}(\kappa,\nu)$ and $\seq{U_\alpha}{\alpha<\kappa}$ is a sequence of dense open subsets of $U$, then there 
 is a $\vec{C}\in\bigcap_{\alpha<\kappa}U_\alpha$ that is not an element of $\mathscr{Triv}(\kappa,\nu)$. Let $\vec{C}_0$ be a $\nu$-coherent $C$-sequence of length $\gamma_0<\kappa$ and $\seq{U_\alpha}{\alpha<\kappa}$ be a sequence of 
 dense open subsets of the set $N_{\vec{C}_0}$ of all extensions of $\vec{C}_0$. 
 In this situation, we can construct a sequence $\vec{C}=\seq{C_\alpha}{\alpha<\kappa}$ and a strictly increasing continuous sequence $\seq{\gamma_\alpha}{\alpha<\kappa}$ of 
 ordinals less than $\kappa$ such that the following statements hold for every $\alpha<\kappa$. 
 \begin{enumerate}
  \item $\seq{C_\beta}{\beta<\gamma_\alpha}$ is a $\nu$-coherent $C$-sequence extending $\vec{C}_0$. 

  \item $N_{\seq{C_\beta}{\beta<\gamma_{\alpha+1}}}$ is a subset of $U_\alpha$. 

  \item $\otp{C_{\gamma_\alpha}}\leq\nu$.
 \end{enumerate}
 Then $\vec{C}$ is a $\nu$-coherent $C$-sequence that is contained in $\bigcap_{\alpha<\kappa}U_\alpha$. 
 Moreover, our construction ensures that $\otp{C_\gamma}\leq\nu$ holds for every element $\gamma$ of the closed unbounded subset $\Set{\gamma_\alpha}{\alpha<\kappa}$ of $\kappa$. 
 Assume, towards a contradiction, that $\vec{C}$ is trivial and let $D$ be a thread through $\vec{C}$ witnessing this. Then there are limit points $\gamma<\delta$ of $C\cap D$ of cofinality $\nu$. 
 By Proposition \ref{proposition:TrivialMuCoherentSequences}, we have $$\nu ~ = ~ \otp{C_\gamma} ~ = ~ \otp{D\cap\gamma} ~ < ~ \otp{D\cap\delta} ~ = ~ \otp{C_\delta} ~ = ~ \nu,$$ a contradiction. 
 This shows that $\mathscr{Triv}(\kappa,\nu)$ is not comeager in any open set.

 To see that the complement of $\mathscr{Triv}(\kappa,\nu)$ is not comeager in any non-empty open set, pick $\vec{C}_0$ and $\seq{U_\alpha}{\alpha<\kappa}$ as above. 
 Again, we construct a sequence $\vec{C}=\seq{C_\alpha}{\alpha<\kappa}$ and a strictly increasing continuous sequence $\seq{\gamma_\alpha}{\alpha<\kappa}$ of 
 ordinals less than $\kappa$ that satisfy the first two of the above statements and the following statement for every $\alpha\in\Lim\cap\kappa$. 
 \begin{enumerate}  
  \item[$\text{(iii)}^\prime$] $C_{\gamma_\alpha} ~ = ~ \Set{\gamma_{\bar{\alpha}}}{\bar{\alpha}<\alpha}$. 
 \end{enumerate}
 If we define $C=\Set{\gamma_\alpha}{\alpha<\kappa}$, then the pair $(\vec{C},C)$ is an element of $\mathscr{Thr}(\kappa,\nu)$ with $\vec{C}\in\bigcap_{\alpha<\kappa}U_\alpha$. 
 Hence the complement of $\mathscr{Triv}(\kappa,\nu)$ is not comeager in any open set.

 Now assume, towards a contradiction, that there is an open subset $U$ of $\mathscr{Coh}(\kappa,\nu)$ and a sequence $\seq{N_\alpha}{\alpha<\kappa}$ of closed nowhere dense subsets of $\mathscr{Coh}(\kappa,\nu)$ 
 such that $\mathscr{Triv}(\kappa,\nu)_\Delta U\subseteq\bigcup_{\alpha<\kappa}N_\alpha$. Then $U$ contains a basic open set $N_{\vec{C}_0}$, because the complement of $\mathscr{Triv}(\kappa,\nu)$ is not comeager 
 in $\mathscr{Coh}(\kappa,\nu)$. Since $\mathscr{Triv}(\kappa,\nu)$ is not comeager in $N_{\vec{C}_0}$, there is a $\vec{C}\in N_{\vec{C}_0}\setminus\mathscr{Triv}(\kappa,\nu)$ such that such that 
 $\vec{C}\notin N_\alpha$ for all $\alpha<\kappa$, a contradiction. 
\end{proof}

This completes the proof of the theorem. 
\end{proof}

\begin{corollary} 
 Let $G$ be $\Add{\kappa}{\kappa^+}$-generic over $\VV$. In $V[G]$, the set $\mathscr{Triv}(\kappa,\nu)$ of trivial $\nu$-coherent $C$-sequences is not a $\mathbf{\Delta}^1_1$-subset of $\mathscr{Coh}(\kappa,\nu)$. 
\end{corollary} 

\begin{proof} 
 This follows from the previous theorem and the fact that $\mathbf{\Delta}^1_1$-subsets of ${}^\kappa\kappa$ have the $\kappa$-Baire property in every $\Add{\kappa}{\kappa^+}$-generic extension of the ground model (see {\cite[Theorem 49]{FHK}}). 
\end{proof}

\begin{proof}[Proof of Theorem \ref{special trees Baire}] 
 Let $\nu$ be an infinite regular cardinal with $\nu = \nu^{{<}\nu}$ and $\kappa=\nu^+=2^\nu$. Assume that every $\kappa$-Aronszajn tree that does not contain a $\kappa$-Souslin subtree is special.
 By Theorem \ref{theorem:THREADSEQNOTBaire}, it suffices to show the set $\mathscr{Triv}(\kappa,\nu)$ is a $\mathbf{\Delta}^1_1$-subset of $\mathscr{Coh}(\kappa,\nu)$ to prove the existence of a $\mathbf{\Delta}^1_1$-subset of ${}^\kappa\kappa$ 
 that does not have the $\kappa$-Baire property.

 We associate to every $C$-sequence $\vec{C}=\seq{C_{\alpha}}{\alpha<\kappa}$ a tree $(\kappa,<_\nu^{\vec{C}})$ defined by $$\alpha<_\nu^{\vec{C}}\beta ~ \Longleftrightarrow ~ \alpha\in\Lim(C_\beta) ~ \wedge ~ \cof{\alpha}=\nu.$$

 \begin{claim*}
  Let $\vec{C}\in\mathscr{Coh}(\kappa,\nu)$. Then $\vec{C}$ is trivial if and only if there is a $\kappa$-branch through the tree $(\kappa,<_\nu^{\vec{C}})$.
 \end{claim*}

 \begin{proof}[Proof of the Claim]
  Let $B\subseteq\kappa$ be a $\kappa$-branch through the tree $(\kappa,<_\nu^{\vec{C}})$ and set $C=\bigcup\Set{C_\alpha}{\alpha\in B}$. Then $C\cap\alpha=C_\alpha$ for every $\alpha\in B$ and this shows that $C$ is a closed unbounded subset of $\kappa$. 
  Since $B$ is a branch, we can conclude that $B$ is equal to the set of all limit points of $C$ of cofinality $\nu$. This shows that $C$ is a $\nu$-thread through $\vec{C}$. 

  In the other direction, let $C$ be a $\nu$-thread through $\vec{C}$ and define $B$ to be the set of all limit points of $C$ of cofinality $\nu$. It follows directly that $B$ is a $\kappa$-branch through the tree $(\kappa,<_\nu^{\vec{C}})$. 
 \end{proof}

 In the following, we say that a $C$-sequence $\vec{C}$ is $\nu$-special if the corresponding tree $(\kappa,<_\nu^{\vec{C}})$ is special. 
 Since special trees have no cofinal branches, the following statement is a direct consequence of the above claim.

 \begin{claim*}
  Let $\vec{C}\in\mathscr{Coh}(\kappa,\nu)$. If $\vec{C}$ is $\nu$-special, then $\vec{C}$ is not trivial. \qed
 \end{claim*}

 We now want to use the above assumption to conclude that the converse of the above implication also holds. Since trees of the form $(\kappa,<_\nu^{\vec{C}})$ might have levels of cardinality $\kappa$, we have to 
 associate a $\kappa$-Aronszajn tree to a given $C$-sequence and apply our assumption to these trees. We will use Todor{\v{c}}evi{\'c}'s method of \emph{minimal walks} to obtain those trees.

 Fix a non-trivial $\nu$-coherent $C$-sequence $\vec{C}=\seq{C_\alpha}{\alpha<\kappa}$. Let $T(\rho_0)$ denote the tree obtained from the \emph{full codes of walks through $\vec{C}$} (see {\cite[Section 1]{MR908147}}). 
 Then $T(\rho_0)$ is a tree of height $\kappa$. By {\cite[1.7]{MR908147}}, the tree  $T(\rho_0)$ has a $\kappa$-branch if and only if there is a closed unbounded subset $C$ of $\kappa$ and $\xi<\kappa$ 
 such that for every $\xi\leq\alpha<\kappa$ there is a $\alpha\leq\beta<\kappa$ with $C\cap\alpha = C_\beta\cap[\xi,\alpha)$. By considering limit points of such a club of cofinality $\nu$, we can conclude that 
 the existence of such a subset would imply that $\vec{C}$ is trivial. Hence there are no $\kappa$-branches through $T(\rho_0)$. Next, {\cite[1.3]{MR908147}} shows that the $\alpha$-th level of $T(\rho_0)$ has cardinality at 
 most $\betrag{\Set{C_\beta\cap\alpha}{\alpha\leq\beta<\kappa}}+\aleph_0$ for every $\alpha<\kappa$. The $\nu$-coherency of $\vec{C}$ and the assumption $\nu=\nu^{{<}\nu}$ allow us to inductively show that 
 this cardinality is at most $\nu$ for all $\alpha<\kappa$. This shows that $T(\rho_0)$ is a $\kappa$-Aronszajn tree. Since {\cite[Corollary 6.3.3]{MR2355670}} says that trees of the form $T(\rho_0)$ do not contain 
 $\kappa$-Souslin subtrees, our assumption implies that the tree $T(\rho_0)$ is special. The obvious adaptation of the argument of the proof of {\cite[Lemma 7.1.7]{MR2355670}} shows that there is a 
 strictly increasing mapping of  $(\kappa,<_\nu^{\vec{C}})$ into $T(\rho_0)$ and we can conclude that $\vec{C}$ is $\nu$-special. We summarize the above computations in the following claim.

 \begin{claim*}
   Let $\vec{C}\in\mathscr{Coh}(\kappa,\nu)$. Then $\vec{C}$ is $\nu$-special if and only if $\vec{C}$ is not trivial. \qed 
 \end{claim*}

 By considering the tree consisting of all pairs $(\vec{D},s)$ such that $\vec{D}=\seq{D_\alpha}{\alpha<\gamma}$ is a $\nu$-coherent $C$-sequences of length $\gamma<\kappa$ and $\map{s}{\gamma}{\nu}$ is a function 
 with $s(\alpha)\neq s(\beta)$ for all $\beta<\gamma$ and $\alpha\in\Lim(C_\beta)$ with $\cof{\alpha}=\nu$, it is easy to see that the set of all $\nu$-special $\nu$-coherent $C$-sequences is a 
 $\mathbf{\Sigma}^1_1$-subset of $\mathscr{Coh}(\kappa,\nu)$. By the above claim, we can conclude that $\mathscr{Triv}(\kappa,\nu)$ is a $\mathbf{\Delta}^1_1$-subset of $\mathscr{Coh}(\kappa,\nu)$. 
\end{proof}

In the remainder of this section, we let $\kappa$ denote an uncountable cardinal with $\kappa=\kappa^{{<}\kappa}$. In order to prove the more general Theorem \ref{non-separable I-kappa sets}, we will construct subsets of ${}^\kappa\kappa$ with the following property that was already used in the proof of Theorem \ref{theorem:THREADSEQNOTBaire}.

\begin{definition} 
 We call a subset $A$ of ${}^{\kappa}\kappa$ \emph{super-dense} if $A\cap(\bigcap_{\alpha<\kappa}U_\alpha)\neq\emptyset$ whenever $\seq{U_\alpha}{\alpha<\kappa}$ is a sequence of dense open subsets of some non-empty open subset of ${}^\kappa\kappa$. 
\end{definition}

The following statement was essentially shown in the proof of Theorem \ref{theorem:THREADSEQNOTBaire}.

\begin{proposition}
 Assume that $A$ and $B$ are disjoint super-dense subsets of ${}^\kappa\kappa$. If $A\subseteq X\subseteq{}^\kappa\kappa\setminus B$, then $X$ does not have the $\kappa$-Baire property. \qed
\end{proposition}

\begin{proof}[Proof of Theorem \ref{non-separable I-kappa sets}] 
 Suppose that $\map{f}{{}^{<\kappa}\kappa\times{}^{<\kappa}\kappa}{\Set{2\cdot \alpha}{\alpha<\kappa}}$ is an injection. 
 Fix a stationary subset  $S$ of $\kappa$ and let $\map{\chi_S}{\kappa}{2}$ denote its characteristic function. 
 Let $\calA_S$ denote the set of all pairs $(x,y)\in {}^{\kappa}\kappa\times{}^\kappa\kappa$ such that the set $C=\Set{\alpha<\kappa}{x(\alpha)=y(\alpha)}$ is a club in $\kappa$ and 
 $$x(\alpha) ~ = ~ f(x\restriction\alpha, ~ y\restriction\alpha)~ + ~ \chi_S(\alpha)$$ for all $\alpha\in C$. Define $\calT_S$ be the tree consisting of all pairs $(s,t)\in {}^{\gamma}\kappa\times{}^\gamma\kappa$ such that $\gamma\in\Lim\cap\kappa$, 
 $D=\Set{\alpha<\gamma}{s(\alpha)=t(\alpha)}$ is a club in $\gamma$ and $$s(\alpha) ~ = ~ f(s\restriction\alpha, ~ t\restriction\alpha) ~ + ~ \chi_S(\alpha)$$ for all $\alpha\in D$. We equip $\calA_S$ with the topology whose basic open sets are of the form 
 $\Set{(x,y)\in\calA}{s\subseteq x,~t\subseteq y}$ with $(s,t)\in\calT_S$.

 \begin{claim*}
  The space $\calA_S$ is homeomorphic to ${}^\kappa\kappa$. 
 \end{claim*}

 \begin{proof}[Proof of the Claim]
  Since we can identify $\calA_S$ with the set of all cofinal branches through the tree $\calT_S$, it suffices to show that the trees $\calT_S$ and ${}^{{<}\kappa}\kappa$ are isomorphic. 
  Given $(s,t)\in\calT_S$ and $\gamma\geq\length{s}$, we can inductively construct an element $(s_*,t_*)$ in $\calT_S$ extending $(s,t)$ such that $\length{s_*}\geq\gamma$ and $s_*(\alpha)\neq t_*(\alpha)$ for all $\alpha\in[\length{s},\gamma)$. 
  This shows that every node in $\calT_S$ has $\kappa$-many direct successors. Since $\calT_S$ is ${<}\kappa$-closed, it follows that the trees $\calT_S$ and ${}^{{<}\kappa}\kappa$ are isomorphic. 
 \end{proof}

 Let $\map{p}{\calA_S}{{}^\kappa\kappa}$ be the projection onto the first coordinate. Define $A_S=p[\calA_S]$.

 \begin{claim*}
  The projection $\map{p}{\calA_S}{{}^\kappa\kappa}$ is continuous. \qed
 \end{claim*}

 \begin{claim*} 
  The projection $\map{p}{\calA_S}{{}^\kappa\kappa}$ is injective. 
 \end{claim*} 

 \begin{proof}[Proof of the Claim]
  Suppose that $(x,y),(x,y^\prime)\in\calA_S$. Let $C$ and $C^\prime$ be the corresponding club subsets of $\kappa$. Fix $\alpha\in C\cap C^\prime$. Then 
  $$f(x\restriction\alpha, ~ y\restriction\alpha) ~ + ~ \chi_S(\alpha) ~ = ~ x(\alpha) ~ = ~ f(x\restriction\alpha, ~ y^\prime\restriction\alpha) ~ + ~ \chi_S(\alpha).$$
  Since $f$ is injective, we have $y\restriction\alpha=y^\prime\restriction\alpha$ for arbitrarily large $\alpha<\kappa$. 
 \end{proof}

 The above claims show that all sets of the form $A_S$ are contained in  the class $I^\kappa$. 
 The next claim implies that disjoint sets of the form $A_S$ and $A_{S^\prime}$ cannot be separated by a set that has the $\kappa$-Baire property.

 \begin{claim*}
  The set $A_S$ is super-dense. 
 \end{claim*} 

 \begin{proof}[Proof of the Claim]
  Let $\seq{U_\alpha}{\alpha<\kappa}$ be a sequence of dense open subsets of some basic open subset $N_{s_0}$ of ${}^\kappa\kappa$. Let $\gamma_0=\length{s_0}$ and 
  define $t_0\in{}^{\gamma_0}\kappa$ by setting $t_0(\alpha)=s_0(\alpha)+1$ for all $\alpha<\gamma_0$. 
  Now we can construct a strictly increasing continuous sequence $\seq{\gamma_\alpha}{\alpha<\kappa}$ and a sequence 
  $\seq{(s_\alpha,t_\alpha)\in{}^{\gamma_\alpha}\kappa\times{}^{\gamma_\alpha}\kappa}{\alpha<\kappa}$ such that the following statements hold for all $\alpha<\kappa$. 
  \begin{enumerate}
   \item If $\bar{\alpha}<\alpha$, then $s_{\bar{\alpha}}=s_\alpha\restriction\gamma_{\bar{\alpha}}$ and $t_{\bar{\alpha}}=t_\alpha\restriction\gamma_{\bar{\alpha}}$. 

   \item $s_{\alpha+1}(\gamma_\alpha) ~ = ~ t_{\alpha+1}(\gamma_\alpha) ~ = ~ f(s_\alpha, ~ t_\alpha) ~ + ~ \chi_S(\alpha)$. 

   \item $N_{s_{\alpha+1}}\subseteq U_\alpha$. 

   \item $t_{\alpha+1}(\beta) ~ = ~ s_{\alpha+1}(\beta)+1$ for all $\beta$ with $\gamma_{\alpha}<\beta<\gamma_{\alpha+1}$.
  \end{enumerate}
If $x=\bigcup_{\alpha<\kappa}s_\alpha$, then $x\in A_S\cap(\bigcap_{\alpha<\kappa}U_\alpha)$. 
\end{proof}

It remains to construct $2^{\kappa}$-many pairwise disjoint sets of the form $A_S$. Let $\seq{S_\alpha}{\alpha<\kappa}$ be a partition of $\kappa$ into $\kappa$-many stationary subsets. 
Given $\emptyset\neq X\subseteq\kappa$, define $S_X=\bigcup_{\alpha\in X}S_\alpha$. Then ${S_X}_\Delta S_Y$ is stationary for all $\emptyset\neq X,Y\subseteq\kappa$ with $X\neq Y$ and 
the above statement follows from the following claim.

\begin{claim*}
  Let  $S$ and $S^\prime$ be stationary subsets of $\kappa$. If $S_\Delta S^\prime$ is stationary in $\kappa$, then the subsets $A_S$ and $A_{S^\prime}$ are disjoint. 
\end{claim*} 

\begin{proof}[Proof of the Claim]
 Suppose that $(x,y)\in\calA_S$ and $(x,y')\in\calA_{S^\prime}$. Let $C$ and $C^\prime$ be the corresponding clubs. 
 Pick $\alpha\in (C\cap C^\prime \cap (S_\Delta S^\prime)$. Then 
 $$f(x\restriction\alpha,~y\restriction\alpha) ~ + ~ \chi_S(\alpha) ~ = ~ x(\alpha) ~ = ~ f(x\restriction\alpha, ~ y\restriction\alpha) ~ + ~ \chi_{S^\prime}(\alpha).$$ 
 But this is a contradiction, because the ordinals $f(x\restriction\alpha,~y\restriction\alpha)$ and $f(x\restriction\alpha, ~ y\restriction\alpha)$ are both even and $\chi_S(\alpha)+\chi_{S^\prime}(\alpha)=1$. 
\end{proof}

This completes the proof of the theorem. 
\end{proof}


\section{Continuous injective images of closed sets}

In this section, we consider images of injective continuous functions $\map{f}{A}{{}^\kappa\kappa}$ for some closed subset $A$ of ${}^\kappa\kappa$. We start by proving Theorem \ref{theorem:SKappa1LinIKappaClosed}.

\begin{proof}[Proof of Theorem \ref{theorem:SKappa1LinIKappaClosed}]
 Suppose that $\varphi(u,v)$ is a $\Sigma_1$-formula, $z\in {}^{\kappa}\kappa$ and $A$ is the set of all $x\in{}^\kappa\kappa$ with $L[x,z]\models \varphi(x,z)$. 
 In the following, we will code models of the form $(L_{\gamma}[x],\in)$ by some (not unique) $y\in{}^{\kappa}\kappa$ with the help of  bijections $\map{f}{L_{\gamma}[x]}{\kappa}$ with $f[\kappa]=\Set{2\cdot\alpha}{\alpha<\kappa}$. 
 This allows us to easily calculate from any $v\subseteq\kappa$ and $y$ whether $v$ is an element of $L_{\gamma}[x]$.

 Let $C$ denote the set of triples $(x, y, \seq{y_n}{n\in\omega})\in {}^{\kappa}\kappa\times {}^{\kappa}\kappa\times {}^{\omega}({}^{\kappa}\kappa)$ such that each $y_n$ codes a minimal well-founded model $M_n$ of \emph{Kripke-Platek set theory $\mathsf{KP}$} 
 and $\anf{\VV=\LL[x,z]}$ such that the following statements hold for all $n<\omega$. 
 \begin{itemize} 
  \item If $i<n$, then $x,y,y_i,z\in M_n$. 
  \item $M_0\models\anf{\textit{$y$ is $<_{\LL[x,z]}$-least element of ${}^\kappa\kappa$ with $\varphi(x,y)$}}$. 
  \item If $n=k+1$, then $M_{k+1}\models\anf{\textit{$y_k$ is the $<_{\LL[x,z]}$-least code for $M_k$}}$. 
 \end{itemize}

 The set $C$ is $\kappa$-Borel, since the conjunction of $M_n\models\mathsf{KP}+\anf{\VV=\LL[x,z]}$ and the remaining conditions is equivalent to an arithmetic sentence about the elements $x$, $y$, $\seq{y_n}{n<\omega}$ and $z$. 
 Then $C$ is an element of the class $I^{\kappa}_{cl}$ by the Lemma \ref{lemma:BorelIsInIKappaClosed}. We have $A=p_0[C]$, where 
 $$\Map{p_0}{{}^{\kappa}\kappa\times {}^{\kappa}\kappa\times {}^{\omega}({}^{\kappa}\kappa)}{{}^{\kappa}\kappa}{(x,y,\seq{y_n}{n<\omega})}{x}.$$ 
 Moreover $p_0\restriction C$ is injective, since every $x\in p_0[C]$ uniquely determines an element $(x,y,\seq{y_n}{n<\omega})$ of $C$. 
\end{proof}

The next observation will allow us to simplify the following proofs.

\begin{proposition}
 Let $A$ be an element of the class $I^\kappa_{cl}$. Then there is a subtree $T$ of ${}^{{<}\kappa}\kappa\times{}^{{<}\kappa}\kappa$ such that $A=p[T]$ and the projection $\map{p}{[T]}{A}$ is injective. 
\end{proposition}

\begin{proof}
 Let $T_0$ be a subtree of ${}^{{<}\kappa}\kappa$ and $\map{f}{[T_0]}{A}$ be a continuous bijection. 
 Fix an enumeration $\seq{u_\alpha}{\alpha<\kappa}$ of ${}^{{<}\kappa}\kappa$. 
Define $T$ to consist of all pairs $(s,t)$ such that $s\in{}^{{<}\kappa}\kappa$, 
 $\seq{u_{t(\alpha)}}{\alpha<\length{s}}$ is a strictly increasing sequence of nodes in $T_0$ and $f[T_{t(\alpha)}]\subseteq N_{s\restriction(\alpha+1)}$ holds for all $\alpha<\length{s}$.

 Pick $y\in A$ and let $x$ denote the unique element of $[T_0]$ with $f(x)=y$. We can construct a strictly increasing sequence $\seq{\beta_\alpha}{\alpha<\kappa}$ of ordinals less than $\kappa$ 
 such that $f[T_{x\restriction\beta_\alpha}]\subseteq N_{y\restriction(\alpha+1)}$ holds for all $\alpha<\kappa$. 
 Pick $z\in{}^\kappa\kappa$ with $x\restriction\beta_\alpha=u_{z(\alpha)}$ for all $\alpha<\kappa$.  
 Then the pair $(y,z)$ is a $\kappa$-branch through $T$. 
 In the other direction, if $(y,z)$ is an element of $[T]$, then $x=\bigcup_{\alpha<\kappa}u_{z(\alpha)}$ is the unique element of $[T_0]$ with $f(x)=y$.
\end{proof}

\begin{lemma}\label{lemma:ConTPreImages}
  The class $I^{\kappa}_{cl}$ is closed under continuous preimages. 
\end{lemma} 
 
\begin{proof} 
 Let $T$ be a subtree of ${}^{{<}\kappa}\kappa\times{}^{{<}\kappa}\kappa$ with $p\restriction[T]$ injective and $\map{f}{{}^{\kappa}\kappa}{{}^{\kappa}\kappa}$ be continuous. 
 Set $B=p[T]$ and $A=f^{-1}[B]$. Define  
 $$C ~ = ~ \Set{(x,y,z)}{(x,y)\in \mathrm{Graph}(f), ~ (y,z)\in [T]}.$$ 
 Then $A=p_0[C]$, $C$ is closed and $p_0\restriction C$ is injective. 
\end{proof}

\begin{remark} 
 The class $I^{\kappa}$ of injective continuous images of ${}^{\kappa}\kappa$ contains no singletons, and therefore not every nonempty continuous preimage of a set in $I^{\kappa}$ is in $I^{\kappa}$. 
\end{remark}

Next, we prove Theorem \ref{Theorem:ClubFilterNotInKappaSolovay} that allows us to construct models that separate the class $I^\kappa_{cl}$ from $\mathbf{\Sigma}^1_1(\kappa)$.

\begin{proof}[Proof of Theorem \ref{Theorem:ClubFilterNotInKappaSolovay}]
 Let $\lambda>\kappa$ be an inaccessible and $G$ be $\Coll{\kappa}{{<}\lambda}$-generic over $\VV$.  
 Assume, towards a contradiction, that there is a subtree  $T\in\VV[G]$ of ${}^{{<}\kappa}\kappa\times{}^{{<}\kappa}\kappa$ such that $\mathrm{Club}_\kappa=p[T]$ and $p\restriction [T]$ is injective in $\VV[G]$. 
  We may assume that $T$ is an element of $\VV$ by passing to a suitable intermediate model.

 Work in $\VV$. Let $\dot{\mathbb{Q}}$ denote the canonical $\Add{\kappa}{1}$-name for the partial order that shoots a club through the $\Add{\kappa}{1}$-generic subset of $\kappa$ by initial segments. 
 It is easy to see that both $\Add{\kappa}{1}*\dot{\QQQ}$ and $\Add{\kappa}{1}*(\dot{\QQQ}\times\dot{\QQQ})$ contain ${<}\kappa$-closed dense subsets and hence these partial orders are forcing equivalent to $\Add{\kappa}{1}$.

 Suppose that $g$ is $\Add{\kappa}{1}$-generic over $\VV$ and $h$ is $\dot{\QQQ}^{g}$-generic over $\VV[g]$ such that there is a $\Coll{\kappa}{{<}\lambda}$-generic filter $H$ over $\VV[g,h]$ with $\VV[g,h,H]=\VV[G]$. 
 Let $\dot{x}$ be the canonical $\Add{\kappa}{1}$-name for the characteristic function of the subset of $\kappa$ induced by the generic filter. 
 By a classical theorem of Silver (see, for example, {\cite[Proposition 7.3]{MR2987148}}), we have $(\HH{\kappa^+}^{\VV[g,h]},\in)\prec_{\Sigma_1}(\HH{\kappa^+}^{\VV[G]},\in)$ and this shows that 
 there is some $y\in({}^\kappa\kappa)^{\VV[g,h]}$ with $(\dot{x}^g,y)\in [T]^{\VV[g,h]}$. Since the partial order $\Add{\kappa}{1}$ is homogeneous, we have 
 $$\eins_{\Add{\kappa}{1}*\dot{\QQQ}}\Vdash \exists y ~ (\dot{x},y)\in [\check{T}].$$

 Now suppose that $g$ is $\Add{\kappa}{1}$-generic over $\VV$ and $h_0\times h_1$ is $(\dot{\QQQ}\times \dot{\QQQ})^{g}$-generic over $\VV[g]$ such that there is a $\Coll{\kappa}{{<}\lambda}$-generic filter $H$ 
 over $\VV[g,h_0,h_1]$ with $\VV[G]=\VV[g,h_0,h_1,H]$. By the above computations, there are $y_0\in\VV[g,h_0]$ and $y_1\in\VV[g,h_1]$ such that $(x,y_i)\in[T]^{\VV[g,h_i]}\subseteq[T]^{\VV[G]}$. 
 Then $y_0=y_1$ and hence $y_0\in\VV[g]$ by mutual genericity. Since $\Add{\kappa}{1}$ is homogeneous, we have  $$\eins_{\Add{\kappa}{1}}\Vdash \exists y ~ (\dot{x},y)\in [\check{T}].$$

 Finally, let $g$ be $\Add{\kappa}{1}$-generic over $\VV$ and $H$ be $\Coll{\kappa}{{<}\lambda}$-generic over $\VV[g]$ with $\VV[G]=\VV[g,H]$. In this situation, the above computations show that there is 
 some $y\in\VV[g]$ with $(x,y)\in [T]^{\VV[g]}\subseteq[T]^{\VV[G]}$ and this shows that $\dot{x}^g$ is an element of $\mathrm{Club}_\kappa$ in $\VV[G]$. Since $(\HH{\kappa^+}^{\VV[g]},\in)\prec_{\Sigma_1}(\HH{\kappa^+}^{\VV[G]},\in)$ 
 holds in this situation, this shows that the subset of $\kappa$ induced by $g$ contains a club subset in $\VV[g]$, a contradiction.

 The proof for $\Add{\kappa}{\kappa^+}$ instead of $\Coll{\kappa}{{<}\lambda}$ is analogous. 
\end{proof}

The next statement is a direct consequence of Theorem \ref{Theorem:ClubFilterNotInKappaSolovay} and Lemma \ref{lemma:ConTPreImages}. 

\begin{corollary} 
 No $\mathbf{\Sigma}^1_1$-complete set (see Example \ref{example:trees}) is contained in the class $I^\kappa_{cl}$ after forcing with either $\Add{\kappa}{\kappa^+}$ or $\Coll{\kappa}{{<}\lambda}$ with $\lambda>\kappa$ inaccessible.   
 In particular, if $\nu$ is regular with $\kappa=\nu^+=2^\nu$, then the set $\mathscr{Triv}(\kappa,\nu)$ of trivial $\nu$-coherent sequences is not a $\mathbf{\Sigma}^1_1$-complete subset of $\mathscr{Coh}(\kappa,\nu)$ in the above forcing extensions. \qed
\end{corollary}

\begin{remark} 
 The results of \cite{MR2987148} show that an arbitrary subset $A$ of ${}^\kappa\kappa$ is contained in the class $I^\kappa_{cl}$ in a forcing extension of the ground model by a ${<}\kappa$-closed partial order $\PPP(A)$ 
 that satisfies the $\kappa^+$-chain condition. Given some enumeration $\seq{s_\alpha}{\alpha<\kappa}$ of ${}^{{<}\kappa}\kappa$ with $\length{s_\alpha}\leq\alpha$ for all $\alpha<\kappa$, 
 the forcing $\PPP(A)$ constructed in {\cite[Section 3]{MR2987148}} adds a subtree $T$ of ${}^{{<}\kappa}2$ and a bijection $\map{f}{A}{[T]}$ 
 such that for every $x\in A$, the branch $f(x)$ is the unique element $y$ of $[T]$ with the property that there exists a $\beta<\kappa$ with $$s_\alpha\subseteq x ~ \Longleftrightarrow ~ y(\goedel{\alpha}{\beta})=1$$ 
 for all $\alpha<\kappa$. 
 In a $\PPP(A)$-generic extension of $\VV$, we define $D$ to consist of all tuples $(s,t,u)$ in ${}^\gamma\kappa\times {}^\gamma 2 \times{}^\gamma \kappa$ with $\gamma$ closed under G\"odel pairing, $t\in T$, $u$ constant with value $\beta$ and 
 $t(\goedel{\alpha}{\beta})=1$ if and only if $s_\alpha\subseteq s$ for all $\alpha<\length{s}$. Let $T_*$ be the subtree of ${}^{{<}\kappa}\kappa\times{}^{{<}\kappa}2\times{}^{{<}\kappa}\kappa$ obtained by the downwards closure of $D$. 
 By the above remarks, the projection onto the first coordinate is a continuous bijection between the subsets $[T_*]$ and $A$. 
\end{remark}


\section{Classes defined from inner models} 

We start this section by proving Theorem \ref{theorem:S1LKappaBaireInSolovay}. This result yields the remaining relations between the classes $S_1^{\LL,\kappa}$, $I^{\kappa}$, $I^\kappa_{cl}$ and $\mathbf{\Sigma}^1_1(\kappa)$ shown in Figure \ref{figure:Results}.

\begin{proof}[Proof of Theorem \ref{theorem:S1LKappaBaireInSolovay}]
 Let $\kappa$ be an uncountable cardinal with $\kappa=\kappa^{{<}\kappa}$, $G$ be $\Coll{\kappa}{{<}\lambda}$-generic over $\VV$ for some inaccessible cardinal $\lambda>\kappa$ and 
 $M$ be an inner model of $\VV[G]$ with $M\subseteq\VV$. We work in $\VV[G]$.

 Let $\varphi(u,v)$ be a $\Sigma_n$-formula, $z\in {}^{\kappa}2$ and $$A ~ = ~ \Set{x\in{}^\kappa\kappa}{M[x,z]\models\varphi(x,z)}.$$
 Suppose that $z_0\in({}^{\kappa}\kappa)^\VV$ codes a bijection between $\kappa$ and ${}^{<\kappa}\kappa$. Let $C$ denote the set of elements of ${}^\kappa\kappa$ that are $\Add{\kappa}{1}$-generic over $M[z,z_0]$ and let $\dot{x}\in M[z,z_0]$ be the canonical 
 $\Add{\kappa}{1}$-name for the generic element of ${}^\kappa\kappa$. Set $$U ~ = ~ \bigcup\Set{N_s}{s\Vdash_{\Add{\kappa}{1}}^{M[z,z_0]} \varphi^{M[\dot{x},\check{z}]}(\dot{x},\check{z})}.$$

 Since we are working in a $\Coll{\kappa}{{<}\lambda}$-generic extension, the set of dense open subsets of $\Add{\kappa}{1}$ contained in $M[z,z_0]$ has cardinality $\kappa$ in $\VV[G]$. Hence $C$ is a comeager subset of ${}^\kappa\kappa$, 
 i.e. it is equal to the intersection of $\kappa$-many dense open subsets of ${}^\kappa\kappa$. 
 We claim that $A_\Delta U\subseteq {}^{\kappa}\kappa\setminus C$. It is easy to see that $C\cap U\subseteq A$. 
 Suppose that  $x\in A\cap C$. Then $M[x,z]\vDash \varphi(x,z)$ and this implies that $$x\restriction\alpha ~ \Vdash^{M[z,z_0]}_{\Add{\kappa}{1}} ~ \varphi^{M[\dot{x},\check{z}]}(\dot{x},\check{z})$$ for some $\alpha<\kappa$. 
 We can conclude that $x$ is an element of $U$ in this case. 
\end{proof}

We present an interesting consequence of Theorem \ref{theorem:S1LKappaBaireInSolovay} and Corollary \ref{corollary: separating I-kappa from S-1}. A different example of such a set was given in \cite{SchlichtThompson}.

\begin{proposition} 
 Suppose that $\VV=\LL[G]$, where $G$ is $\Coll{\kappa}{{<}\lambda}$-generic over $\LL$ with $\lambda>\kappa$ inaccessible in $\LL$. Then the class $S_1^{\LL,\kappa}$ contains a proper $\mathbf{\Sigma}^1_1$-set which is not $\mathbf{\Sigma}^1_1$-complete. 
\end{proposition} 

\begin{proof} 
 Let $A=\Set{(x,y)\in{}^\kappa\kappa\times{}^\kappa\kappa}{y\in\LL[x]}$. To show that the complement of $A$ is not a $\mathbf{\Sigma}^1_1$-subset in $\LL[G]$, suppose that $x\in A$ if and only if $\varphi(x,x_0)$ holds in $\HH{\kappa^+}$, 
 where $\varphi(u,v)$ is a $\Sigma_1$-formula and $x_0\in\HH{\kappa^+}$. Then there is some $x\in {}^{\kappa}\kappa$ such that $x_0\in\LL[x]$ and $(\HH{\kappa^+}^{\LL[x]},\in)\prec_{\Sigma_1}(\HH{\kappa^+},\in)$. 
 Since ${}^\kappa\kappa\nsubseteq\LL[x]$, it follows that $\exists y\in{}^\kappa\kappa ~ \varphi(x,y)$ holds in $\HH{\kappa^+}$ and there is a $z\in({}^\kappa\kappa)^{\LL[x]}$ such that $\varphi(x,z)$ holds in $\HH{\kappa^+}$. 
 But this implies $z\notin\LL[x]$ by the choice of $\varphi(u,v)$, a contradiction.

 By coding continuous functions $\map{f}{{}^\kappa\kappa}{{}^\kappa\kappa}$ as elements of ${}^\kappa\kappa$, it is easy to see that the class $S_1^{\LL,\kappa}$ is closed under continuous preimages. 
 This shows that the set $A$ is not $\mathbf{\Sigma}^1_1$-complete, because otherwise every $\mathbf{\Sigma}^1_1$-subsets of ${}^\kappa\kappa$ would be contained in the class $S_1^{\LL,\kappa}$, contradicting Corollary \ref{corollary: separating I-kappa from S-1}.  
\end{proof}

By Proposition \ref{proposition: Sigma-1-1 and S-1 are equal in L}, the classes $S_1^{\LL,\kappa}$ and $\mathbf{\Sigma}^1_1(\kappa)$ coincide in $\LL$.  
In contrast, it also consistent that  for a fixed $n<\omega$,  the class $S_n^{\LL,\kappa}$ consists of $\mathbf{\Delta}^1_1$-subsets.

\begin{remark}
 Given $0<n<\omega$, it is possible to modify coding techniques developed in \cite{holylocally} to construct a forcing extension of the ground model $\VV$ in which the class $S_n^{\LL,\kappa}$ consists of $\mathbf{\Delta}^1_1(\kappa)$ subsets of ${}^\kappa\kappa$. 
 In this construction, we fix a universal $\Sigma_n$-formula $\Phi(v_0,v_1)$ and recursively construct a notion of forcing $\PPP$ with the property that the set 
 $$Sat_n^\LL ~ = ~ \Set{(n,x)\in\omega\times({}^\kappa\kappa)^{\VV[G]}}{\Phi(n,x)^{\LL[x]}}$$ is definable over the structure $(\HH{\kappa^+}^{\VV[G]},\in)$ by a $\Sigma_1$-formula with parameters in every $\PPP$-generic extension $\VV[G]$ of $\VV$. 
 This is possible, because the validity of the statement $\Phi(n,x)^{\LL[x]}$ is absolute between $\VV[G]$ and any intermediate extension given by some complete subforcing of $\PPP$. 
 It is easy to see that the $\mathbf{\Sigma}^1_1(\kappa)$-definability of $Sat_{n+1}^\LL$ implies that every subset in $S_n^{\LL,\kappa}$ is $\mathbf{\Delta}^1_1(\kappa)$-definable. 
\end{remark}

Sets contained in the class $S_1^{\LL,\kappa}$ are always $\mathbf{\Sigma}^1_1$-definable. We close this section by showing that for arbitrary inner models $M$, sets in the class $S_1^{M,\kappa}$ are not necessarily definable from ordinals and subsets of $\kappa$.

\begin{proposition}
 Let $\mu>\kappa$ be a cardinal and $G\times H$ be $(\Add{\kappa}{\mu}\times\Add{\kappa}{\mu})$-generic over $\VV$. 
 In $\VV[G,H]$, there is a set $A$ contained in the class $S_1^{\VV[G],\kappa}$ which is not definable from ordinals and subsets of $\kappa$.  
\end{proposition} 

\begin{proof} 
 Let $\mu= X_0\sqcup X_1$ with $\betrag{X_0}=\betrag{X_1}=\mu$. We may assume that $G\times H$ is $(\PPP\times\QQQ)$-generic over $\VV$, where $\PPP=\Add{\kappa}{X_0}$ and $\QQQ=\Add{\kappa}{X_1}$. 
 Given $x\in {}^{\kappa}2$, define $s(x)=\Set{\alpha<\kappa}{s(\alpha)=1}$. Define $A$ to be the set of all $x\in{}^\kappa 2$ such that $s(x)$ contains a club in $\VV[G,x]$.

 Assume, towards a contradiction, that there is a formula $\psi(v_0,v_1,v_2)$ that defines $A$ using parameters $a\in {}^{\kappa}\kappa$ and $\gamma\in\On$. 
 Suppose that $\dot{a}\in\VV$ is a nice $(\PPP\times\QQQ)$-name with $\dot{a}^{G\times H}=a$. Then there is a condition $p\in G\times H$ which forces over $\VV$ that $A$ is defined by $\psi(\cdot,\dot{a},\check{\gamma})$. 
 Let $S=\mathrm{supp}(p)\cup \mathrm{supp}(\dot{a})$. Since $\betrag{S}\leq\kappa$, we can find $\alpha\in X_0\setminus S$, $\beta\in X_1\setminus S$ and an automorphism $\pi$ of $\mathbb{P}\times\mathbb{Q}$ which switches only the coordinates $\alpha$ and $\beta$.

 Let $\dot{\QQQ}$ denote the canonical $\Add{\kappa}{1}$-name for a partial order that shoots a club through the generic subset of $\kappa$ by initial segments. Then $\Add{\kappa}{1}*\dot{\QQQ}$ is forcing equivalent to $\Add{\kappa}{1}$. 
 Given $\delta<\mu$, $\dot{x}_{\delta}$ and $\dot{y}_{\delta}$ be canonical $(\PPP\times\QQQ)$-names such that $\dot{x}_\delta^{\bar{G}\times\bar{H}}$ is the generic subset induced by $g$ and $\dot{y}_\delta^{\bar{G}\times\bar{H}}$ is the 
 corresponding generic club subset whenever $\bar{G}\times\bar{H}$ is $(\PPP\times\QQQ)$-generic over $\VV$, $\bar{G}_\delta$ is the filter on $\Add{\kappa}{1}$ induced by the $\delta$-th component of $\bar{G}$ and $g*h$ is the filter in 
 $\Add{\kappa}{1}*\dot{\QQQ}$ induced by $\bar{G}_\delta$.

 Since $\alpha\in X_0$, we have $\dot{x}_\alpha^{G\times H},\dot{y}_\alpha^{G\times H}\in\VV[G]$ and hence $\dot{x}_\alpha^{G\times H}\in A$. 
 By the homogeneity of $\PPP\times\QQQ$, we have $p\Vdash\psi(\dot{x}_{\alpha},\dot{a},\check{\gamma})$. 
 By applying $\pi$, we get $p\Vdash \psi(\dot{x}_{\beta}, \dot{a},\check{\gamma})$ and $\dot{x}_\beta^{G\times H}$ is an element of $A$. 
 But this yields a contradiction, because $\dot{x}_{\beta}^{G\times H}$ is $\Add{\kappa}{1}$-generic over $\VV[G]$ and hence $\dot{x}_{\beta}^{G\times H}$ does not contain a club subset in $\VV[G,\dot{x}_{\beta}^{G\times H}]$. 
\end{proof}


\section{Trees of higher cardinalities}\label{section:HigherCardTrees}

The following observation shows that the class of projections of ${<}\kappa$-closed trees of height $\kappa$ without end nodes increases if we consider subtrees of ${}^{<\kappa}\kappa\times{}^{<\kappa}(\kappa^+)$, i.e. the class $C^\kappa$ is always a proper subclass of the class $C^{\kappa,\kappa^+}$.

\begin{lemma}\label{lemma:WIsKappaPlusProjection}
 The closed subset $W$ defined by equation \emph{(\ref{equation:DefinitionOfW})} in Section \ref{Section:thin} is the projection of a ${<}\kappa$-closed subtree of ${}^{{<}\kappa}\kappa\times{}^{{<}\kappa}(\kappa^+)$ without end nodes. 
\end{lemma} 

\begin{proof} 
 Let $D$ denote the set of all pairs $(s,t)$ in ${}^{\gamma}2\times {}^{\gamma}(\kappa^+)$ such that $\gamma<\kappa$ is closed under G\"odel pairing and $\map{t}{\gamma}{\kappa^+}$ is an injection such that for all $\alpha,\beta<\gamma$, $\alpha<_s \beta$ 
 if and only if $t(\alpha)< t(\beta)$, where the relation $<_s:=R_s$ is defined as in Section \ref{Section:thin}. Then $D$ is a  ${<}\kappa$-closed subset of ${}^{{<}\kappa}\kappa\times{}^{{<}\kappa}(\kappa^+)$.

\begin{claim*}
 Every element of $D$ is extended by an element of $[D]$. In particular,  $D$ has no end nodes. 
\end{claim*} 

\begin{proof}[Proof of the Claim]
 Suppose that $(s,t)\in D$ with $\length{s}=\gamma$. We extend $s$ to $x\in {}^{\kappa}2$ such that $x(\goedel{\alpha}{\beta})=1$ is equivalent to $\alpha<\beta$ for all $\alpha,\beta<\kappa$ with $\beta\geq\gamma$ and extend $t$ to $y\in {}^{\kappa}(\kappa^+)$ such that $\ran{t}\subseteq y(\gamma)$ and $y\restriction[\gamma,\kappa)$ is strictly increasing. Then $(x,y)$ is an element of $[D]$. 
\end{proof}

In this situation, Lemma \ref{projection to projection of tree} implies that there is a ${<}\kappa$-closed subtree $T$ of ${}^{{<}\kappa}\kappa\times{}^{{<}\kappa}(\kappa^+)$ with $W=p[D]=p[T]$. 
\end{proof}

The following result shows that it is consistent that the classes $C^{\kappa,\mu}$ and $C^{\kappa,\mu}_{cl}$ coincide for some $\mu<2^\kappa$. In particular, there can be a $\mu<2^\kappa$ such that every $\mathbf{\Sigma}^1_1$-subset is an element of $C^{\kappa,\mu}$.

\begin{lemma}\label{lemma:AddForcesMuKappa}
 Let $\kappa$ be an uncountable cardinal with $\kappa=\kappa^{{<}\kappa}$, $\mu=2^\kappa$ and $G$ be $\Add{\kappa}{\theta}$-generic over $\VV$ for some cardinal $\theta$. In $\VV[G]$, every closed subset of ${}^\kappa\mu$ is equal to a continuous image of ${}^\kappa\mu$.  
\end{lemma} 

\begin{proof} 
 We may assume that $\theta>\mu$, because otherwise $\mu=\betrag{{}^\kappa\mu}^{\VV[G]}$ and the statement of the lemma holds trivially. Let $A=[T]^{\VV[G]}$ be a closed subset of ${}^\kappa\mu$ in $\VV[G]$. 
 Since $\Add{\kappa}{\theta}$ satisfies the $\kappa^+$-chain condition and we can identify $T$ with a subset of $\mu$, we may assume that $T$ is an element of $\VV$. By Proposition \ref{proposition:ClosedRetract}, it suffices to show that $A$ 
is equal to the projection of a ${<}\kappa$-closed subtree of ${}^{{<}\kappa}\mu\times{}^{{<}\kappa}\mu$ without end nodes in $\VV[G]$.

Let $\sigma$ be an $\Add{\kappa}{\theta}$-nice name for an element of $[T]$ in $\VV$. Since $\Add{\kappa}{\theta}$ satisfies the $\kappa^+$-chain condition, 
the set $X=\supp{\sigma}$ is a subset of $\theta$ of cardinality at most $\kappa$. Define $T_{\sigma}$ to be the subtree of ${}^{{<}\kappa}\mu\times{}^{{<}\kappa}\Add{\kappa}{X}$ that consists of all pairs 
$(t,\vec{p}\hspace{1.2pt})\in{}^\gamma\kappa\times{}^\gamma\Add{\kappa}{X}$ such that $\gamma<\kappa$ and the following statements hold. 
\begin{itemize}
 \item $\seq{\vec{p}\hspace{1.2pt}(\alpha)}{\alpha<\gamma}$ is a descending sequence of conditions in $\Add{\kappa}{X}$. 
 \item If $\alpha<\gamma$, then $\vec{p}\hspace{1.2pt}(\alpha)\Vdash^V_{\Add{\kappa}{X}}\anf{\check{s}\restriction (\check{\alpha}+1)\subseteq \sigma}$. 
\end{itemize}
Then $[T_\sigma]$ is a ${<}\kappa$-closed tree without end notes. We have $\sigma^G\in p[T_\sigma]^{\VV[G]}$ and $p[T_\sigma]^{\VV[G]}\subseteq A$, because $A$ is closed.

We can find an automorphism $\pi$ of $\Add{\kappa}{\theta}$ in $\VV$ with $\supp{\pi(\sigma)}\subseteq\kappa$. If $(t,\seq{\vec{p}\hspace{1.2pt}(\alpha)}{\alpha<\gamma})$ is an element of $T_\sigma$, then it follows that 
$(t,\seq{\pi(\vec{p}\hspace{1.2pt}(\alpha))}{\alpha<\gamma})$ is an element of $T_{\pi(\sigma)}$. This shows that $p[T_\sigma]^{\VV[G]}=p[T_{\pi(\sigma)}]^{\VV[G]}$.

The above computations show that for every element $x$ of $A$ there is a $\Add{\kappa}{\kappa}$-nice name $\sigma$ 
for an element of $[T]$ in $\VV$ such that $x\in p[T_\sigma]^{\VV[G]}\subseteq A$. Since there are at most $\mu$-many $\Add{\kappa}{\kappa}$-nice names for elements of $[T]$ in $\VV$ and $\Add{\kappa}{\kappa}$ has cardinality $\kappa$ in $\VV$, 
we can conclude that $A$ is equal to the union of $\mu$-many projections of ${<}\kappa$-closed subtrees of ${}^{{<}\kappa}\mu\times{}^{{<}\kappa}\kappa$ without end nodes. 
But this shows that $A$ is equal to the projection of a ${<}\kappa$-closed subtree of ${}^{{<}\kappa}\mu\times{}^{{<}\kappa}\mu$ without end nodes. 
\end{proof}


\section{Perfect embeddings}\label{section:PerfectEmb}

 In this section, we isolate a property of subtrees of ${}^{{<}\kappa}\kappa$ that implies that the corresponding closed subset of ${}^\kappa\kappa$ is not a continuous image of a space of the form ${}^\kappa\mu$. 
 We will use this implication to prove several consistency results mentioned in Subsection \ref{subsection:ContImagesHigherCard}.

 \begin{theorem}\label{theorem:SurjImageWideLevel}
  Let $\kappa$ be an uncountable regular, $U$ be an unbounded subset of $\kappa$ and $T$ be a subtree of ${}^{{<}\kappa}\kappa$. 
 If $\mu$ is a cardinal with $\mu^{{<}\kappa}<\betrag{[T]}$ and $\map{c}{{}^\kappa\mu}{[T]}$ is a continuous surjection, then there is a strictly increasing sequence $\seq{\lambda_n\in U}{n<\omega}$ with least upper bound $\lambda$ and an injection
  \begin{equation*}
    \map{i}{\prod_{n<\omega}\lambda_n}{T(\lambda)}.
  \end{equation*}
 such that 
 \begin{equation*}
    x\restriction n = y\restriction n ~ \Longleftrightarrow ~ i(x)\restriction\lambda_n = i(y)\restriction\lambda_n
 \end{equation*}
 holds for all $x,y\in\prod_{n<\omega}\lambda_n$ and $n<\omega$. 
 \end{theorem}

\begin{proof}
 We start by proving two claims.

 \begin{claim*}
  If $U$ is an open subset of ${}^\kappa\nu$ with $\betrag{c[U]}>\mu^{{<}\kappa}$, then there is an $x\in U$ with $\betrag{c[N_{x\restriction\alpha}]}>\mu^{{<}\kappa}$ for all $\alpha<\kappa$. 
 \end{claim*}

 \begin{proof}[Proof of the Claim]
  Assume, towards a contradiction, that for every $x\in U$ there is an $\alpha_x<\kappa$ with $\betrag{c[N_{x\restriction\alpha_x}]}\leq\mu^{{<}\kappa}$. 
  Define $A=\Set{x\restriction\alpha_x}{x\in U}\subseteq{}^{{<}\kappa}\mu$. Then 
  \begin{equation*}
   \betrag{c[U]} ~ \leq ~ \betrag{\bigcup_{t\in A} c[N_t]} ~ \leq ~ \sum_{t\in A}\betrag{c[N_t]} ~ \leq ~ \mu^{{<}\kappa}, 
  \end{equation*}
  a contradiction.
 \end{proof}

\begin{claim*}
  Let $\bar{\gamma},\lambda<\kappa$ and $s\in{}^{{<}\kappa}\mu$ with $\betrag{c[N_s]}>\mu^{{<}\kappa}$. 
  Then there is an ordinal $\bar{\gamma}\leq\gamma<\kappa$ and a sequence $\seq{y_\alpha\in N_s}{\alpha<\lambda}$ 
  such that the following statements hold for all $\alpha,\bar{\alpha}<\lambda$.
  \begin{enumerate}
   \item $\betrag{c[N_{y_\alpha\restriction\beta}]}>\mu^{{<}\kappa}$ for all $\beta<\kappa$. 

   \item $c(y_\alpha)\restriction\gamma=c(y_{\bar{\alpha}})\restriction\gamma$ if and only if $\alpha=\bar{\alpha}$.
  \end{enumerate}
 \end{claim*}

\begin{proof}[Proof of the Claim]
  We recursively construct 
  \begin{itemize}
   \item a strictly increasing sequence $\seq{\gamma_\alpha}{\alpha<\lambda}$ of ordinals and

   \item a sequence $\seq{y_\alpha\in N_s}{\alpha<\lambda}$
  \end{itemize}
  such that the following statements hold for all $\alpha<\lambda$.
 \begin{enumerate}
  \item[(a)] $\betrag{c[N_{y_\alpha\restriction\beta}]}>\mu^{{<}\kappa}$ for all $\beta<\kappa$. 

  \item[(b)] If $\bar{\alpha}<\alpha$, then $c(y_\alpha)\restriction\gamma_\alpha\neq c(y_{\bar{\alpha}})\restriction\gamma_\alpha$.
 \end{enumerate}

 Assume that the above sequences are constructed up to $\alpha<\lambda$. We define 
 \begin{equation*}
  U ~ = ~ N_s\setminus c^{-1}[\Set{c(y_{\bar{\alpha}})}{\bar{\alpha}<\alpha}].
 \end{equation*}
 Then $U$ is open and $\betrag{c[U]}>\mu^{{<}\kappa}$. In this situation, the above claim implies that there is a $y_\alpha\in U$ such that the statement (a) holds. 
 Since $c(y_\alpha)\neq c(y_{\bar{\alpha}})$ for all $\bar{\alpha}<\alpha$, there is a $\bar{\gamma}\leq\gamma<\kappa$ with $\gamma\geq\lub_{\bar{\alpha}<\alpha}\gamma_{\bar{\alpha}}$ and $c(y_{\bar{\alpha}})\notin N_{c(y_\alpha)\restriction\gamma_\alpha}$ 
 for all $\bar{\alpha}<\alpha$. This implies that the above statement (b) holds. 

 If we define $\gamma=\sup_{\alpha<\lambda}\gamma_\alpha<\kappa$, then this construction ensures that the statement of the lemma holds.
 \end{proof}

 Now we recursively construct  
  \begin{itemize}
   \item a strictly increasing sequence $\seq{\lambda_n<\kappa}{n<\omega}$ of ordinals, 

   \item a sequence $\seq{s_p\in{}^{{<}\kappa}\mu}{n<\omega,~p\in\prod_{m<n}(\lambda_m+1)}$ of functions, and 

   \item a sequence $\seq{t_p\in T(\lambda_n)}{n<\omega,~p\in\prod_{m<n}(\lambda_m+1)}$ of nodes in $T$
  \end{itemize}
  such that the following statements hold for all $n<\omega$ and $p,q\in\prod_{m<n}(\lambda_m+1)$.
  \begin{enumerate}
   \item $0<\lambda_{n+1}\in U$. 

   \item $\betrag{c[N_{s_p}]}>\mu^{{<}\kappa}$ and $c[N_{s_p}]\subseteq N_{t_p}$.

   \item $N_{t_p}\cap N_{t_q}\neq\emptyset$ if and only if $p=q$. 

   \item If $m<n$, then $s_{p\restriction m}\subsetneq s_p$ and $t_{p\restriction m}=t_p\restriction\lambda_m$.
  \end{enumerate}

  Set $s_\emptyset=t_\emptyset=\emptyset$ and $\lambda_0=0$. Now assume that $n<\omega$ and the sequences $\seq{\lambda_m}{m\leq n}$, $\seq{s_p}{p\in\prod_{m<n}(\lambda_m+1)}$ and $\seq{t_p}{p\in\prod_{m<n}(\lambda_m+1)}$ are already constructed. 
 Pick $p\in\prod_{m<n}(\lambda_m+1)$. Apply the second claim to $\lambda_n$, $\lambda_n+1$ and $s_p$ to obtain an ordinal $\lambda_n\leq\gamma_p<\kappa$ and a sequence $\seq{y_\alpha^p\in N_{s_p}}{\alpha<\lambda_n+1}$ satisfying the properties (i) and (ii) stated in the claim. Define 
  \begin{equation*}
   \lambda_{n+1} ~ = \min(U\setminus\lub\Set{\gamma_p}{p\in\prod_{m<n}(\lambda_m+1)}) ~ < ~ \kappa.
  \end{equation*}
  Given $p\in\prod_{m<n}(\lambda_m+1)$ and $\alpha<\lambda_n+1$, we define 
  \begin{equation*}
   t_{p^\frown\langle\alpha\rangle} ~ = ~ c(y^p_\alpha)\restriction\lambda_{n+1}
  \end{equation*}
  and
  \begin{equation*}
   s_{p^\frown\langle\alpha\rangle} ~ = ~ y^p_\alpha\restriction\min\Set{\beta<\kappa}{c[N_{y^p_\alpha\restriction\beta}]\subseteq N_{t_{p^\frown\langle\alpha\rangle}}}.
  \end{equation*}

Let $\lambda=\sup_{n<\omega}\lambda_n$. Fix a $z\in\prod_{n<\omega}(\lambda_n+1)$. Then we have $s_{z\restriction n}\subseteq s_{z\restriction(n+1)}$ and $t_{z\restriction n}\subseteq t_{z\restriction(n+1)}$ for all $n<\omega$. Hence there is an $x_z\in\bigcap_{n<\omega}N_{s_{z\restriction n}}$ and our construction yields  
 \begin{equation*}
  c(x_z) ~ \in ~ [T] ~ \cap ~ \bigcap_{n<\omega}N_{t_{z\restriction n}}.
 \end{equation*}

 Pick $z_0,z_1\in\prod_{n<\omega}(\lambda_n+1)$ and $n<\omega$.  Assume $z_0\restriction (n+1)\neq z_1\restriction(n+1)$. Our construction ensures $N_{t_{z_0\restriction(n+1)}}\cap N_{t_{z_1\restriction(n+1)}}=\emptyset$ and $c(x_{z_i})\in N_{t_{z_i\restriction(n+1)}}$ for all $i<2$. Hence $c(x_{z_0})\restriction\lambda_{n+1}\neq c(x_{z_1})\restriction\lambda_{n+1}$.  
In the other direction, if $z_0\restriction n=z_1\restriction n$, then $c(x_{z_0})\in N_{t_{z_0\restriction n}}\cap N_{t_{z_1\restriction n}}\neq\emptyset$ and we get $z_0\restriction n=z_1\restriction n$. In particular, we can conclude that the induced map 
 \begin{equation*}
  \Map{i_0}{\prod_{n<\omega}(\lambda_n+1)}{T(\lambda)}{z}{c(x_z)\restriction\lambda}
 \end{equation*}
 is an injection. By embedding the product $\prod_{0<n<\omega}\lambda_n$ into $\prod_{n<\omega}(\lambda_n+1)$,   the statement of the theorem follows. 
\end{proof}

We use Theorem \ref{theorem:SurjImageWideLevel} to prove the following result that will imply Theorem \ref{theorem:ManipulateCardinalCharC} and Theorem \ref{theorem:TreesFromCohenReal}.

\begin{theorem}\label{theorem:TreesFromCoverReals}
 Assume that there is an inner model $M$ of $\ZFC$ such that $\RRR\nsubseteq M$ and 
 $M$ has the $\omega_1$-cover property (see \cite{MR2063629}), i.e. every countable set of ordinals in $\VV$ is covered by a set that is an element of $M$ and countable in $M$. 
 If $\kappa$ is an uncountable regular cardinal, then there is a closed subset $A$ of ${}^\kappa\kappa$ such that $A$ is not a continuous image of ${}^\kappa\mu$ 
 for every cardinal $\mu$ with $\mu^{{<}\kappa}<\betrag{(2^\kappa)^M}^\VV$. 
\end{theorem}

The following result due to Veli{\v{c}}kovi{\'c} and Woodin will be the key ingredient of this proof. Remember that a tree $T$ is \emph{superperfect} if above every node of $T$ there is a node with infinitely many direct successors.

\begin{theorem}[{\cite[Theorem 2]{MR1640916}}]\label{theorem:VelickovicWoodin}
  Let $T$ be a superperfect subtree of ${}^{{<}\omega}\omega$. If $M$ is an inner model of $\ZFC$ with $[T]^\VV\subseteq M$, then $\RRR^\VV\subseteq M$. 
\end{theorem}

\begin{proof}[Proof of Theorem \ref{theorem:TreesFromCoverReals}]
 Let $M$ be an inner model of $\ZFC$ with $\RRR^\VV\nsubseteq M$ and the property that every countable set of ordinals in $\VV$ is covered by a set that is contained in $M$ and countable in $M$. 
We work in $\VV$. Fix an uncountable regular cardinal $\kappa$ and let $T$ denote the subtree $({}^{{<}\kappa}2)^M$ of ${}^{{<}\kappa}\kappa$. Assume, towards a contradiction, that there is a continuous surjection $\map{c}{{}^\kappa\mu}{[T]}$ for some cardinal $\mu$ with $\mu^{{<}\kappa}<\betrag{(2^\kappa)^M}$. Since $T$ has at least $\betrag{(2^\kappa)^M}$-many cofinal branches, we can use Theorem \ref{theorem:SurjImageWideLevel} to find a strictly increasing sequence $\seq{\lambda_n<\kappa}{n<\omega}$ with limit $\lambda$ and an injection $\map{i}{{}^\omega\omega}{T(\lambda)}$ with 
\begin{equation*}
 x\restriction n = y\restriction n ~ \Longleftrightarrow ~ i(x)\restriction \lambda_n = i(y)\restriction\lambda_n
\end{equation*}
for all $x,y\in{}^\omega\omega$ and $n<\omega$. 
By our assumptions, we have $\cof{\lambda}^M=\omega$ and there is a strictly increasing sequence $\seq{\eta_n}{n<\omega}$ contained in $M$ that is cofinal in $M$. Moreover, with the help of a well-ordering of $T$ in $M$ we can find $C\in M$ such that $C$ is countable in $M$ and the set  $\Set{i(x)\restriction\eta_n}{x\in{}^\omega\omega, ~ n<\omega}$ is contained in $C$. Define $$B ~ = ~ \Set{t\in T(\lambda)}{\forall n<\omega ~ t\restriction\eta_n\in C} ~ \in ~ M.$$ Then $\ran{i}\subseteq B$.

\begin{claim*}
 If $t\in{}^\lambda 2$ such that for every $k<\omega$ there is an $x\in{}^\omega\omega$ satisfying  $t\restriction\eta_k=i(x)\restriction\eta_k$, then $t\in\ran{i}\subseteq T(\lambda)\subseteq M$. 
\end{claim*}

\begin{proof}[Proof of the Claim]
 Fix $n<\omega$. Then there is a $k<\omega$ with $\eta_k\geq\lambda_n$ and our assumption gives us an  $x_n\in{}^\omega\omega$ with $t\restriction\lambda_n=i(x_n)\restriction\lambda_n$. Define $s_n=x_n\restriction n$. Given $x\in{}^\omega\omega$, the properties of $i$ imply that $i(x)\restriction\lambda_n=t\restriction\lambda_n$ if and only if $s_n\subseteq x$. This shows that $x_t=\bigcup_{n<\omega}s_n$ is an element of ${}^\omega\omega$ with $i(x_t)=t$. 
\end{proof}

Fix an injective enumeration $\seq{c_n}{n<\omega}$ of $C$ that is contained in $M$ and let $\map{j}{B}{({}^\omega\omega)^M}$ denote the injection in $M$ defined by $$j(t)(n)=k ~ \Longleftrightarrow ~ t\restriction\eta_n=c_k.$$

\begin{claim*}
 The set $\ran{j\circ i}$ is closed in ${}^\omega\omega$.
\end{claim*}

\begin{proof}[Proof of the Claim]
 Let $z$ be a limit point of $\ran{j\circ i}$ in ${}^\omega\omega$. 
 Fix $k\leq n<\omega$. Then there is an $x\in{}^\omega\omega$ with $(j\circ i)(x)\restriction(n+1)=z\restriction(n+1)$ and $$c_{z(k)} ~ = ~  c_{(j\circ i)(x)(k)} ~ = ~ i(x)\restriction \eta_k ~ \subseteq ~ i(x)\restriction\eta_n ~ = ~ c_{(j\circ i)(x)(n)} ~ = ~ c_{z(n)} ~ \in ~ T(\eta_n).$$ This shows that $t=\bigcup_{n<\omega}c_{z(n)}$ is an element of ${}^\lambda 2$ with the property that for every $n<\omega$ there is an $x_n\in{}^\omega\omega$ with $t\restriction\eta_n=i(x_n)\restriction\eta_n$. By the above claim, there is an $x\in{}^\omega\omega$ with $i(x)=t$.  If $k<\omega$, then $i(x)\restriction\eta_k=t\restriction\eta_k=c_{z(k)}$. This shows that $ z=(j\circ i)(x)\in\ran{j\circ i}$. 
\end{proof}

Let $S=\Set{(j\circ i)(x)\restriction n}{x\in{}^\omega\omega,~ n<\omega}$. Then $[S]=\ran{j\circ i}\subseteq M$.

\begin{claim*}
 $S$ is a superperfect subtree of ${}^{{<}\omega}\omega$.
\end{claim*}

\begin{proof}[Proof of the Theorem]
 Let $x\in{}^\omega\omega$ and $k<\omega$. Then there are $m,n<\omega$ such that $\eta_k\leq\lambda_m<\lambda_{m+1}\leq\eta_n$. Then there is a sequence $\seq{x_p\in{}^\omega\omega}{p<\omega}$ such that $x_p\restriction m = x\restriction m$ and $x_p(m)=p$. Given $p<\omega$, we have $i(x)\restriction\eta_l =  i(x_p)\restriction\eta_l$ for every $l<k$ and therefore $(j\circ i)(x)\restriction k \subseteq (j\circ i)(x_p)\restriction(n+1)$. Since $$(j\circ i)(x_p)\restriction(n+1)\neq(j\circ i)(x_q)\restriction(n+1)$$ for all $p<q<\omega$, this shows that $(j\circ i)(x)\restriction k$ has infinitely many successor at level $n+1$. But this shows that there is a node above $(j\circ i)(x)\restriction k$ in $S$ of length at most $n$ that has infinitely many direct successors. 
\end{proof}

The combination of the last claim and Theorem \ref{theorem:VelickovicWoodin} shows that $\RRR\subseteq M$, a contradiction. 
\end{proof}

We close this section with the proofs of the theorems mentioned above.

\begin{proof}[Proof of Theorem \ref{theorem:TreesFromCohenReal}]
 Assume that $\VV$ is an $\Add{\omega}{1}$-generic extension of some ground model $M$. Then $M$ satisfies the assumptions of Theorem \ref{theorem:TreesFromCoverReals}, because $\Add{\omega}{1}$ satisfies the countable chain condition 
 and adds a new real. Let $\kappa$ be an uncountable regular and $A$ be the closed subset of ${}^\kappa\kappa$ given by Theorem \ref{theorem:TreesFromCoverReals}. If $\mu$ is a cardinal with $\mu^{{<}\kappa}<2^\kappa$, then 
 $\mu^{{<}\kappa}<(2^\kappa)^M$ and $A$ is not equal to a continuous image of ${}^\kappa\mu$. 
\end{proof}

\begin{proof}[Proof of Theorem \ref{theorem:ManipulateCardinalCharC}]
 Fix an uncountable cardinal $\kappa$ with $\kappa=\kappa^{{<}\kappa}$, a cardinal $\mu\geq2^\kappa$ with $\mu=\mu^\kappa$ and a cardinal $\theta\geq\mu$ with $\theta=\theta^\kappa$. 
 Let $G$ be $\Add{\kappa}{\mu}$-generic over $\VV$, $H$ be $\Add{\omega}{1}$-generic over $\VV[G]$ and $K$ be $\Add{\kappa}{\theta}^{\VV[G,H]}$-generic over $\VV[G,H]$. 
 Then $\mu=(2^\kappa)^{\VV[G]}=(2^\kappa)^{\VV[G,H]}$, $\kappa=(\kappa^{{<}\kappa})^{\VV[G,H]}$, $\theta=(2^\kappa)^{\VV[G,H,K]}$ and all models have the same cardinals. 
 
 We work in $\VV[G,H,K]$. Then the inner model $\VV[G]$ satisfies the assumptions of Theorem \ref{theorem:TreesFromCoverReals} and there is a closed subset $A$ of ${}^\kappa\kappa$ such that $A$ 
 is not equal to a continuous image of ${}^\kappa\bar{\mu}$ for any cardinal $\bar{\mu}$ with $\bar{\mu}^{{<}\kappa}<\betrag{(2^\kappa)^{\VV[G]}}=\mu$. Finally, Lemma \ref{lemma:AddForcesMuKappa} 
 implies that every closed subset of ${}^\kappa\mu$ is equal to a continuous image of ${}^\kappa\mu$.  
\end{proof}


\section{Questions}

 We close this paper with questions raised by the above results.

 Proposition \ref{proposition: Sigma-1-1 and S-1 are equal in L} shows that the classes $\mathbf{\Sigma}^1_1(\kappa)$ and $S^{\LL.\kappa}_1$ coincide in models of the form $\LL[x]$ with $x\subseteq\kappa$. 
This induces the question whether there are other models in which these classes are identical. 

\begin{question}
If the classes $\mathbf{\Sigma}^1_1(\kappa)$ and $S_1^{\LL,\kappa}$ coincide, is there an $x\subseteq\kappa$ with ${}^\kappa\kappa\subseteq\LL[x]$?
\end{question}

 By Proposition \ref{proposition: Sigma-1-1 and S-1 are equal in L} and Theorem \ref{Theorem:ClubFilterNotInKappaSolovay}, the statement \anf{\emph{the club filter is a continuous injective image of a closed subset of ${}^\kappa\kappa$}} 
 is not absolute between different models of set theory. 
 In particular, it is consistent that the club filter is not a continuous injective image of the whole space ${}^\kappa\kappa$. Therefore it is natural to ask whether the negation of this statement is also consistent.

\begin{question} 
Is it consistent that the club filter $\mathrm{Club}_{\kappa}$ on $\kappa$ is an element of $I^{\kappa}$? 
\end{question}

 The discussion presented in Subsection \ref{subsection:ContImagesHigherCard} shows that it is consistent that all subsets in $C^{\kappa,\kappa^+}$ can consistently be $\mathbf{\Sigma}^1_1$-definable. 
 Therefore we may ask if these continuous images can consistently be contained in the smaller classes considered in this paper.

\begin{question}
 Is it consistent that the class $C^{\kappa,\kappa^+}$ is contained in the class $I^\kappa_{cl}$?
\end{question}


\bibliographystyle{alpha}
\bibliography{references}

\end{document}